\documentclass[a4paper, 11pt] {article}
\usepackage{amsfonts, amsmath, amssymb, amsthm,url,dsfont}
\usepackage{mathtools}
\usepackage{graphicx}
\usepackage[english] {babel}
\usepackage{enumitem}

\usepackage{array}
\usepackage{cellspace}

\usepackage{geometry}
\newtheorem{thm}{Theorem}[section]
\newtheorem{cor}[thm]{Corollary}

\newtheorem{prop}[thm]{Proposition}

\theoremstyle{definition}
\newtheorem{rem}[thm]{Remark}
\newtheorem{ex}[thm]{Example}

\newcommand{\R}{\mathbb{R}}

\DeclareMathOperator{\nor}{nor}

\newcommand{\eps}{\varepsilon}

\newcommand{\G}{\mathcal{G}}

\DeclareMathOperator{\SO}{SO}

\DeclareMathOperator{\sgn}{sgn}
\newcommand{\di}{\mathrm{d}}
\newcommand{\Ha}{\mathcal{H}}
\newcommand{\Li}{\mathcal{L}}

\begin{document}

\title{Rotational Crofton formulae for Minkowski tensors \\
and some affine counterparts}

\author{Anne Marie Svane and Eva B. Vedel~Jensen}


\maketitle

\begin{abstract}
Motivated by applications in local stereology, a new rotational  
Crofton formula is derived for Minkowski tensors. For sets of positive reach, the formula shows how rotational averages of intrinsically defined Minkowski tensors on sections passing through the origin are related to the geometry of the sectioned set. In particular, for Minkowski tensors of order $j-1$ on $j$-dimensional linear subspaces, we derive an explicit formula for the rotational average involving hypergeometric functions. Sectioning with lines and  hyperplanes through the origin is considered in detail. We also study the case where the sections are not restricted to pass through the origin. For sets of positive reach, we here obtain  a Crofton formula for the integral mean of intrinsically defined Minkowski tensors on $j$-dimensional affine subspaces.
\end{abstract}

\section{Introduction}

In local stereology, rotational averages of measurements on sections through fixed points are considered \cite{jensen}. Local stereology is applied in optical microscopy  which allows virtual sections to be generated through reference points in a tissue block. A typical example is optical sectioning through the nucleus of a biological cell. A technical advantage of such sectioning in biological material is that the boundary of a central section is often much more clearly visible than the boundary of a peripheral section. Local stereology is by now recognized as being a very powerful tool in biomedicine, especially in neuroscience and cancer grading. 

Motivated by applications in local stereology, we study in the present paper functionals $\Psi$ defined on a set of positive reach $ X$ by
\begin{equation}
\Psi(X)=\int_{\mathcal{L}_j^d}\Psi_L(X\cap L)\,\di L,
\label{rotationformula}
\end{equation}
where $\mathcal{L}_j^d$ is the space of all $j$-dimensional linear subspaces in $\mathbb{R}^d$, $\Psi_L$ is a functional on the sets $X\cap L$ and $\mathrm{d}L$ is the element of the rotation invariant measure on $\mathcal{L}_j^d$. Such functionals arise in local stereology where $\Psi_L(X\cap L)$ is observed on a random section $L$, distributed according to a normalized version of the rotation invariant measure. Then, \eqref{rotationformula} implies that the mean of $\Psi_L(X\cap L)$ is, up to a known constant, equal to $\Psi(X)$.   

For the fundamental case where $\Psi_L$ is one of the intrinsic volumes, an explicit expression for $\Psi(X)$ was determined under weak regularity conditions in \cite{evarataj}, see also \cite{auneau}. In particular, explicit expressions for $\Psi(X)$ are available in the case where $\Psi_L$ is volume and surface area in $L$. In the present paper, we will generalize this result by applying \eqref{rotationformula} to a general class of functionals $\Psi_L$ that contains Minkowski tensors on $L$ of arbitrary rank. The intrinsic volumes are Minkowski tensors of rank 0. Minkowski tensors of rank 1 or higher contain important information about position, shape and orientation. The particular case where $\Psi_L$ is a volume tensor was treated in \cite[(10) and Proposition 5.3]{jeremy}.  In the present paper, we treat explicitly the case of surface tensors. 

We also study functionals of the form
\begin{equation}
\Psi(X)=\int_{\mathcal{E}_j^d}\Psi_E(X\cap E)\,\mathrm{d}E,
\label{motionformula}
\end{equation}
where $\mathcal{E}_j^d$ is the space of all $j$-dimensional affine subspaces in $\mathbb{R}^d$, $\Psi_E$ is a functional on the sets $X\cap E$ and $\mathrm{d}E$ is the element of the motion invariant measure on $\mathcal{E}_j^d$.

Functionals of the form (\ref{motionformula}) are well studied in the literature. In the case where $\Psi_E$ is an intrinsic volume, $\Psi$ is again an intrinsic volume. This is the classical Crofton formula. The case of Minkowski tensors has been treated in \cite{schuster}  while very general formulae are derived in \cite{rataj}. Recently, Hug and Weis \cite{weis} have studied the case where $\Psi_E(X\cap E)$ is substituted by a tensor-valued measure.  

In the present paper, we consider the case where $\Psi_E$ is an arbitrary Minkowski tensor. Using the basic theorems in \cite{rataj}, we show for an arbitrary  set $X$ of positive reach that the functional $\Psi$ in \eqref{motionformula} is a linear combination of Minkowski tensors. The same formulae are obtained for total measures in the Crofton formulae for tensor-valued measures, derived for convex bodies in \cite{weis}. 

The paper is structured as follows. In Section \ref{notation}, definitions and basic notation used for Grassmann manifolds, generalized curvature measures, Minkowski tensors, and hypergeometric functions are shortly summarized. The rotational integral formulae of the type \eqref{rotationformula} are derived in Section \ref{rotationsec}, while some affine counterparts of the type \eqref{motionformula} may be found in Section \ref{affinesec}. Proofs are deferred to an Appendix.

\section{Notation and definitions}\label{notation}
We first introduce some relevant notation and definitions that we are going to use throughout the paper. 

\subsection{Grassmann manifolds}
Let $\Li_{j}^d$ denote the Grassmannian consisting of $j$-dimensional linear subspaces of $\R^d$. The measure we consider on $\Li_{j}^d$ is the rotation invariant measure, which is unique up to a constant. More specifically, the measure is the $j(d-j)$-dimensional Hausdorff measure  on $\Li_{j}^d$ considered as a subspace of the vector space $\bigwedge_j \R^d$ by identifying $L\in \Li_{j}^d$ with $v_1\wedge \dotsm \wedge v_j$ for any oriented orthonormal basis $v_1,\dots,v_j$ spanning $L$, see \cite[Chapter 1]{federer}. This measure has total  measure given by 
\begin{equation*}
c_{d,j}:=\Ha^{j(d-j)}(\Li_{j}^d) = \frac{\sigma_d \dotsm \sigma_{d-j+1}}{\sigma_j \dotsm \sigma_1},
\end{equation*}
see \cite[3.2.28]{federer}. Here, $\Ha^m$ denotes the $m$-dimensional Hausdorff measure and 
\begin{equation*}
\sigma_k = 2\pi^{{k}/{2}}/\Gamma({k}/{2})=\Ha^{k-1}(S^{k-1})
\end{equation*}
 is the surface area of the $(k-1)$-dimensional sphere. By convention $c_{d,0}=1$.

For $L \in \Li_{j-1}^d$ and $x\notin L$, we let $L^x\in \Li_j^d$ denote the linear subspace spanned by $L$ and $x$. We let $p(x| L)$ be the orthogonal projection of $x$ onto $L$ and $\pi(x| L)=p(x| L)/|p(x| L)|\in S^{d-1}$ its normalization. Similarly, if $v$ is a vector, we write $p(x|v)$ for the projection of $x$ onto the line through the origin spanned by  $v$ and $\pi(x|v)$ for its normalization. If $L\subseteq \R^d$ is a linear subspace of $\R^d$ of dimension larger than $j$, then $\Li_j^L$ denotes the space of $j$-dimensional linear subspaces of $L$. If $v\in \R^d$ is a non-zero vector, then $\Li_j^v$ denotes the set of $j$-dimensional linear subspaces of $\R^d$ containing $v$.

Given two subspaces $L_j\in \Li_j^d$ and $L_k\in \Li_k^d$, we define the generalized sine function $\G(L_j, L_k)$ as follows. An orthonormal basis for $L_j\cap L_k$ is extended to an orthonormal basis for $L_j$ and one for $L_k$. Then, $\G(L_j, L_k)$ is the volume of the parallelepiped spanned by all these vectors.
In particular, $\G(L_{d-1}, L_k)=|p(n| L_k)|$, where $n$ is a unit normal of $L_{d-1}$.

Let $v\in \R^d$ and assume $v\neq 0$. Consider the function 
\begin{equation*}
h: \Li_j^d \backslash \{L\in \Li_j^d \mid v\perp L\}  \to S^{d-1}
\end{equation*}
mapping $L$ to $\pi(v | L)$.
Then, the $(d-1)$-Jacobian (see \cite{federer}) was computed for $L$ with $v\notin L$ and $v\notin L^\perp $ in \cite[Lemma 4.2]{rataj} to be
\begin{equation*}
J_{d-1}h(L)= \bigg(\frac{|p(v | L^\perp)|}{|p(v | L)|} \bigg)^{j-1}.
\end{equation*}
This allows us to apply the coarea formula to a bounded measurable function $f : \Li_j^d \to \R$ as follows
\begin{equation}\label{coarea}
\int_{\Li_j^d} f(L)\, \di L =  \int_{S^{d-1}}\mathds{1}_{\{\langle u,v \rangle > 0\}} \bigg(\frac{1-\langle u,v \rangle^2}{\langle v, u \rangle^2 } \bigg)^{\frac{1-j}{2}} \int_{\Li_{j-1}^{v^\perp \cap u^\perp}} f(M^u)\, \di M\, \di u.
\end{equation}
Here, $\di L$ denotes the element of the rotation invariant measure on $\Li_j^d$ while $\di u$ is the element of the $(d-1)$-dimensional Hausdorff measure on $S^{d-1}$.
We will also make repeated use of the following integration formula, see e.g.\ \cite[Lemma 1.3.1]{groemer},
\begin{equation}\label{sphereform}
\int_{S^{d-1}} f(u)\, \di u = \int_{S^{d-1}\cap v^\perp} \int_{-1}^1 f\big(tv+\sqrt{1-t^2}w\big)(1-t^2)^{\frac{d-3}{2}}\,\di t\, \di w,
\end{equation}
where $v\in S^{d-1}$ is any unit vector.

Finally, we let $\mathcal{E}_j^d$ denote the affine Grassmannian consisting of $j$-dimensional affine subspaces of $\R^d$. The element of the motion invariant measure on $\mathcal{E}_j^d$ is denoted by $\di E$ where for $E=L+y$ with $y\in L^\perp$ we have $\di E = \di y \, \di L$.
If $L\subseteq \R^d$ is a $j$-dimensional linear subspace, we denote by $S^{j-1}(L)$ the unit sphere in  $L$ and $\Sigma^L = L\times S^{j-1}(L)$. Similarly, if $E=L+y$ is an affine subspace, we write $S^{j-1}(E)=S^{j-1}(L)$ and $\Sigma^E=E\times S^{j-1}(E)$. 

\subsection{Generalized curvature measures}
The reach of a closed set $X\subseteq \R^d$ is the supremum of all $R$ satisfying that every point $x$ at distance less than $R$ from $X$ has a unique closest point in $X$. We denote this closest point by $p_X(x)\in X$. 
The space of non-empty compact sets in $\R^d$ (resp. $E\in \mathcal{E}_j^d$) having positive reach will be denoted by  $\mathcal{PR}^d$ (resp. $\mathcal{PR}(E)$). Similarly, let $\mathcal{K}^d$ (resp. $\mathcal{K}(E)$) denote the set of non-empty compact convex subsets of $\R^d$ (resp. $E\in \mathcal{E}_j^d$).  All convex sets have infinite reach, so $\mathcal{K}^d \subseteq \mathcal{PR}^d$.

For $X\in \mathcal{PR}^d$, the generalized curvature measures $\Lambda_k(X;\cdot)$, $k=0,\dots,d-1$, are measures on $\Sigma$. They were introduced for sets of positive reach in \cite{zahle}, see also \cite{schneider} in the case of convex sets, and they satisfy the following local Steiner formula
\begin{equation*}
\Ha^d\Big(x\in \R^d \mid 0 < d(x,X)<\eps, \Big(p_X(x),\tfrac{x-p_X(x)}{|x-p_X(x)|}\Big) \in A\Big) = \sum_{k=0}^{d-1} \eps^{d-k} \kappa_{d-k} \Lambda_k(X; A),
\end{equation*}
for any Borel set $A\subseteq \Sigma$ and $\eps$ smaller than the reach of $X$. The so-called intrinsic volumes are obtained as $V_k(X)=\Lambda_k(X;\Sigma)$, $k=0,\dots,d-1$.

The generalized curvature measures can be described explicitly as follows. For $X\in \mathcal{PR}^d$, the unit normal bundle $\nor X$ of $X$ is the set of support elements, i.e.\ the set of pairs $(x,n)$ for which $x$ is a boundary point of $X$ and $n$ is an outer unit normal of $X$ at $x$. More specifically,
\begin{equation*}
\nor X = \big\{\big(x, \tfrac{y-x}{|y-x|} \big) \in \R^d \times S^{d-1} \mid y\notin X,\, p_X(y)=x  \big\}.
\end{equation*}
This is a $(d-1)$-rectifiable set. A basis for the tangent space of $\nor X$ at $(x,n)$ is given by the vectors
\begin{equation}\label{awedge}
 \Big( \tfrac{1}{\sqrt{1+\kappa_i(x,n)^2}} a_i(x,n), \tfrac{\kappa_i(x,n)}{\sqrt{1+\kappa_i(x,n)^2}} a_i(x,n)\Big), \qquad i=1,\dots,d-1,
\end{equation}
where $a_i(x,n)$ are the principal directions at $(x,n)$ corresponding to the principal curvatures $\kappa_i(x,n)$, $i=1,\dots,d-1$. 
Integration of a locally bounded measurable function $\psi: \Sigma \to \R$ with respect to $\Lambda_k(X;\cdot)$ is then given by
\begin{align}\label{kappa}
\int_{\Sigma} \psi(x{}&,n)\, \Lambda_k (X;\di (x,n) ) \\
{}&= \frac{1}{\sigma_{d-k}}\int_{\nor X} \psi(x,n) \sum_{|I|=d-k-1}  \frac{\prod_{i\in I} \kappa_i(x,n)}{\prod_i \sqrt{1+\kappa_i(x,n)^2}}\, \Ha^{d-1}(\di (x,n)).\nonumber
\end{align}

If $X\in \mathcal{PR}(E)$ for some $E\in \mathcal{E}_j^d$, there are also generalized curvature measures relative to $E$, denoted $\Lambda_k^E(X;\cdot)$. These are measures on $\Sigma^E$ satisfying the analogue of the local Steiner formula in $E$
\begin{equation*}
\Ha^j\Big(x\in E \mid 0 < d(x,X)<\eps, \Big(p_X(x),\tfrac{x-p_X(x)}{|x-p_X(x)|}\Big) \in A\Big) = \sum_{k=0}^{j-1} \eps^{j-k} \kappa_{j-k}\, \Lambda_k^E(X; A),
\end{equation*}
where $A\subseteq \Sigma^E$ is a Borel set and $\eps$ is smaller than the reach of $X$.

We are going to consider families of valuations $\Psi^\psi_{k,E}$ on $\mathcal{PR}(E)$ of the following form. For $X \in \mathcal{PR}(E)$, 
\begin{align}\label{psiE}
\Psi^\psi_{k,E}(X)= {}& \int_{\Sigma^E} \psi(E,x,n)\, \Lambda_k^E(X;\di (x,n)),
\end{align}
where $\psi: \mathcal{U}_j^d \to \R$ is a  function on
\begin{equation}\label{Uaff}
\mathcal{U}_j^d =\big\{(E,x,n)\in \mathcal{E}_j^d \times \R^{d} \times  S^{d-1} \mid (x,n) \in \Sigma^E\big\}.
\end{equation}
To ensure integrability, we assume that $\psi$ is measurable and locally bounded.

\subsection{Minkowski tensors}\label{minkowskisec}
We are particularly interested in a special case of \eqref{psiE}, known as the Minkowski tensors. To define these, we let $\mathbb{T}^p$ be the vector space of symmetric tensors of rank $p\in \mathbb{N}_0$ on $\R^d$. The volume tensors are defined for $X\in \mathcal{PR}^d$ and $p\in \mathbb{N}_0 $ by
\begin{equation*}
\Phi_{d}^{p,0}(X) = \int_X x^p\, \di x \in \mathbb{T}^p,
\end{equation*}
where $x^p$ is the tensor product of $p$ copies of $x$. The integration is to be understood coordinatewise. The integral geometry of volume tensors is well understood \cite{jeremy,schuster,ali,jens}, so this paper will focus on the remaining Minkowski tensors. These are defined for $r,s\in \mathbb{N}_0$ and $0\leq k \leq d-1$ as follows
\begin{equation}\label{defsurftens}
\Phi_{k}^{r,s}(X) =  \frac{\sigma_{d-k}}{r!s!\sigma_{d-k+s}} \int_{\Sigma} x^r n^s\, \Lambda_k(X;\di (x,n)) \in \mathbb{T}^{r+s},
\end{equation}
where $x^r n^s$ denotes the symmetric tensor product of $r$ copies of $x$ and $s$ copies of $n$. The tensors in \eqref{defsurftens} are sometimes called surface tensors. Using \eqref{kappa} coordinatewise, we also have 
\begin{equation*}
\Phi_k^{r,s}(X) = \frac{1}{r!s!\sigma_{d-k+s}}\int_{\nor X} x^r n^s \sum_{|I|=d-k-1}  \frac{\prod_{i\in I} \kappa_i(x,n)}{\prod_i \sqrt{1+\kappa_i(x,n)^2}}\, \Ha^{d-1}(\di (x,n)).
\end{equation*}
If $X\in \mathcal{PR}(E)$, we can replace $\Sigma$ and $\Lambda_k(X;\cdot)$ by  $\Sigma^E$ and $\Lambda_k^E(X;\cdot)$ in \eqref{defsurftens}. The resulting tensors are thus defined relative to $E$ (i.e.\ intrinsically defined) and are denoted by $\Phi_{k,E}^{r,s}(X)$. 

In the literature, Minkowski tensors are usually only considered for $X\in \mathcal{K}^d$, but since both the definition and the results of this paper hold for sets of positive reach, satisfying  mild regularity conditions, we will be working in this generality.

We let $Q\in \mathbb{T}^2$ denote the metric tensor $Q=\sum_{i=1}^d v_i^2$, where $v_1,\dots,v_d$ is an orthonormal basis of $\R^d$. Similarly, for $L\in \Li_j^d$ we define the metric tensor on $L$ by $Q(L)=\sum_{i=1}^j w_i^2$, where $w_1,\dots,w_j$ is any orthonormal basis for $L$.

The Minkowski tensors, considered as functionals on $\mathcal{K}^d$ with values in $\mathbb{T}^{r+s}$, have the following properties:
\begin{itemize}
\item[(i)] Continuity with respect to the Hausdorff metric on $\mathcal{K}^d$.
\item[(ii)] Isometry covariance: $\Phi_{k}^{r,s}(X+t) = \sum_l \varphi_{r+s-l}(X)t^l$ for all $t\in \R^d$, and $\Phi_{k}^{r,s}(\theta X) =\theta \Phi_{k}^{r,s}(X) $ for any rotation $\theta\in \SO(d)$ (see \cite{hugsch} for details).  
\item[(iii)] Valuation property: If $X_1,X_2, X_1\cup X_2\in \mathcal{K}^d$, then
\begin{equation*}
\Phi_{k}^{r,s}(X_1) + \Phi_{k}^{r,s}(X_2)= \Phi_{k}^{r,s}(X_1\cup X_2) +  \Phi_{k}^{r,s}(X_1\cap X_2). 
\end{equation*}
\end{itemize}
 According to Alesker's classification theorem \cite{alesker}, all tensor-valued functionals with the properties (i)--(iii) are linear combinations of the tensors $Q^l\Phi_{k}^{r,s}(X)$. 

The Minkowski tensors can be viewed as the total measures of the tensor valued measures given for $X\in \mathcal{PR}^d$ on  a Borel set $A\subseteq \Sigma$ as follows
\begin{equation*}
\Phi_{k}^{r,s}(X;A) =  \frac{\sigma_{d-k}}{r!s!\sigma_{d-k+s}} \int_{\Sigma}\mathds{1}_{\{(x,n)\in A\}} x^r n^s\, \Lambda_k(X;\di (x,n)).
\end{equation*}
These measures are called the local Minkowski tensors.
In the classification of local tensor valuations on $\mathcal{K}^d$, some new tensor measures $\Phi_{k}^{r,s,1}$ with very similar properties were discovered \cite{hugschneider,hugsch}. These are the so-called generalized local Minkowski tensors given by (\cite[(2.38)]{hugsch})
\begin{align*}
\Phi_{k}^{r,s,1}(X;A)= \frac{1}{r!s!\sigma_{d-k+s}}\int_{\nor X\cap A} x^rn^s \sum_{|I|= d-k-1} \frac{\prod_{i\in I}\kappa_i}{\prod_i \sqrt{1+\kappa_i^2}}\sum_{i\notin I} a_i(x,n)^2\, \Ha^{d-1}(\di (x,n))
\end{align*}
for $k\in \{1,\dots,d-1\}$, $r,s\geq 0$, and $A\subseteq \Sigma$ a Borel set. We let $\Phi_{k}^{r,s,1}(X)=\Phi_{k}^{r,s,1}(X;\Sigma)$. 
Although the local Minkowski tensors $Q^l\Phi_{k}^{r,s}$, $2l+r+2=p$, and the generalized local Minkowski tensors $Q^l\Phi_{k}^{r,s,1}$, $2l+r+s+2=p$, are linearly independent, there are linear dependences between their total measures, as the following proposition shows.
\begin{prop}\label{linkombprop}
Let $X\in \mathcal{PR}^d$, $r\geq 0$, and $s\geq 2$. Then 
\begin{equation}\label{trivialcase}
\Phi_{d-1}^{r,s-2,1}(X) = Q\Phi_{d-1}^{r,s-2}(X) - 2\pi s  \Phi_{d-1}^{r,s}(X).
\end{equation}
For $1\leq k\leq d-2$,
\begin{equation}\label{linkomb}
\Phi_{k}^{r,s-2,1}(X) = \sum_{l=0}^{s-1}  2\pi (s-1-l)\Phi_{k-l-1}^{r+l+1,s-l-1}(X) - Q \sum_{l=0}^{s-3} \Phi_{k-l-1}^{r+l+1,s-l-3}(X) 
\end{equation}
and 
\begin{equation}\label{linkomb2}
\Phi_{k}^{r,s-2,1}(X) = Q\sum_{l=0}^{r} \Phi_{k+l}^{r-l,s-2+l}(X)-  \sum_{l=0}^{r}  2\pi (s+l)\Phi_{k+l}^{r-l,s+l}(X). 
\end{equation}
In particular, 
\begin{equation}\label{r0}
\Phi_{k}^{0,s-2,1}(X) = Q \Phi_{k}^{0,s-2}(X) - 2\pi  s\Phi_{k}^{0,s}(X). 
\end{equation}
\end{prop}
For $X\in \mathcal{K}^d$, the results in Proposition \ref{linkombprop}  were observed in \cite[Remark 4.1]{hugschneider} (referring to computations in  \cite{schuster}). In the Appendix, it is shown that the definition of $\Phi_{k}^{r,s,1}$ makes sense and Proposition \ref{linkombprop} holds more generally for sets of positive reach.

We end the discussion of tensors by defining the contraction of two tensors $T\in \mathbb{T}^{r+s}$ and $S \in \mathbb{T}^r$ as follows: If $S=v_1\odot \dotsm \odot v_r$, where $\odot$ denotes the symmetric tensor product, then the contraction $\text{Contr}(T,S)$ of $T$ and $S$ is an element of $\mathbb{T}^s$ given by
\begin{equation*}
\text{Contr}(T,v_1\odot \dotsm \odot v_r )(\cdot) = T (v_1,\dots , v_r, \cdot ),
\end{equation*}
where $T$ is identified with its dual map $(\R^d)^{r+s}\to \R$. This is extended to all $S$ by linearity.

\subsection{Hypergeometric functions}
Hypergeometric functions show up in many of the formulae below. We therefore recall some basic definitions and properties here. More information can be found in \cite{slater} or \cite{wolfram}.
The hypergeometric function ${}_pF_q$ has $p+q$ parameters $a_1,\dots,a_p,b_1,\dots, b_q\in \R$ and is given by the power series expansion 
\begin{align*}
{}_pF_q(a_1,\dots,a_p;b_1,\dots,b_q;z) = \sum_{n=0}^\infty \frac{\prod_{i=1}^p(a_i)_n }{\prod_{i=1}^q(b_i)_n}\frac{z^n}{n!}, \quad z\in \R,
\end{align*}
where the Pochhammer symbol is defined by
\begin{equation*}
(a)_n =\frac{\Gamma(a+n)}{\Gamma(a)}=a\cdot (a+1)\dotsm (a+n-1)
\end{equation*}
when $n$ is a positive integer and $(a)_0=1$.
We shall only need the case $p=q+1$. Then ${}_pF_q$ has convergence radius at least 1 and converges absolutely at $z=1$ if $\sum_{i}a_i -\sum_i b_i >0$. If some $b_i \leq 0 $ is an integer (and $b_i$ is maximal among $b_1,\dots,b_q$ with this property), then ${}_pF_q$ is undefined unless there is an integer $0\geq a_j\geq b_i$, in which case we define 
\begin{align*}
{}_pF_q(a_1,\dots,a_p;b_1,\dots,b_q;z) = \sum_{n=0}^{-a_j} \frac{\prod_1^p(a_i)_n }{\prod_i^q(b_i)_n}\frac{z^n}{n!}.
\end{align*}
(This interpretation  seems to be non-standard when $a_j=b_i$, but we include this case to simplify notation later).

The most important case is $p=2$ and $q=1$, where we have the following integral representation for $0<b<c$
\begin{equation*}
{}_2F_1(a,b;c;z) = \frac{\Gamma(c)}{\Gamma(b)\Gamma(c-b)}\int_0^1 (1-zt)^{-a} t^{b-1}(1-t)^{c-b-1}\, \di t. 
\end{equation*}
We will also need 
Gauss's hypergeometric theorem 
\begin{equation}\label{value1}
{}_2F_1(a,b;c;1) = \frac{\Gamma(c)\Gamma(c-a-b)}{\Gamma(c-a)\Gamma(c-b)},
\end{equation}
which holds whenever $c>a+b$.

\section{Rotational Crofton formulae}\label{rotationsec}
\subsection{A general rotational formula}
In this section, we consider rotational integrals of the form
\begin{equation}\label{PsiL}
\Psi(X)=\int_{\Li_{j}^d} \Psi^\psi_{k,L} (X\cap L)\, \di L, 
\end{equation}
where $X\in \mathcal{PR}^d$ is a set of positive reach, $\Psi^\psi_{k,L}$ is a functional on $\mathcal{PR}(L)$ of the form \eqref{psiE}, and $0\leq k<j<d$. 

We will restrict ourselves to the class $\widetilde{\mathcal{PR}}^d$ consisting of sets $X$ of positive reach, satisfying: 
\begin{itemize}
\item[(i)] $o \notin \partial X$. 
\item[(ii)] For almost all $L\in \mathcal{L}_d^j$, there is no $(x,n)\in \nor X$ with $x\in L$ and $n$  perpendicular to $x$. 
\end{itemize}
According to \cite[Theorem 4.10]{federer2}, the condition (ii) ensures that $X\cap L$ has positive reach for almost all $L$  and hence the integrand in \eqref{PsiL} is defined almost surely. The condition (i) is discussed in Remark \ref{boundary0} below.
The restriction to $\widetilde{\mathcal{PR}}^d$ is rather mild. It was thus shown in \cite[Proposition 1]{evarataj} that the class $\widetilde{\mathcal{PR}}^d$ contains all convex sets $X$ satisfying $o\notin \partial X$. Furthermore, if $X\in \mathcal{PR}^d$, then almost all translations of $X$ will belong to $\widetilde{\mathcal{PR}}^d$. 

Theorem \ref{rotational} below shows that the integral in \eqref{PsiL} exists for all $X\in \widetilde{\mathcal{PR}}^d$. Moreover, the theorem gives an explicit formula for $\Psi(X)$. 
In the special case where $\psi$ is a function of $x$ only, such a formula was already given in \cite[p. 558]{evarataj}. 
To state the theorem, we introduce for $(x,n)\in \nor X$ the notation $A_I(x,n)$ for the tangent subspace
\begin{equation*}
A_I(x,n) = \text{span}\{a_i(x,n),i\notin I\},
\end{equation*}
where $I\subseteq \{1,\dots,d-1\}$ and $a_i(x,n)$, $i=1,\dots,d-1$, are the principal directions. 

\begin{thm} \label{rotational}
Suppose $X\in \widetilde{\mathcal{PR}}^d $. Let $\psi: \mathcal{U}_j^d \to \R$ be a locally bounded measurable function and  $0\leq k<j<d$. Then, 
\begin{align}\label{mainrotform}
\int_{\Li_{j}^d} \Psi^\psi_{k,L} (X\cap L)\, \di L ={}& \frac{1}{\sigma_{j-k}}\int_{\nor X} \frac{1}{|x|^{d-j}}\sum_{|I|=j-1-k} \frac{\prod_{i\in I} \kappa_i(x,n)}{\prod_{i=1}^{d-1} \sqrt{1+\kappa_i(x,n)^2}} \\
&\times \int_{\Li_{j-1}^{x^\perp}} \psi\big(L^x,x,\pi\big(n | L^x\big)\big) \frac{\G\big(L^x, A_I(x,n)\big)^2}{ \big|p(n | L^x)\big|^{j-k}}\, \di L\, \Ha^{d-1}(\di (x,n)). \nonumber
\end{align}
In particular, the integral on the left hand side exists.
\end{thm}

A proof of Theorem \ref{rotational} can be found in the Appendix. The proof follows the lines of \cite{evarataj}, but avoids the theory of slices. Instead, the area and coarea formulae are applied directly. 

\begin{rem}\label{boundary0}
Theorem 3.1 does not hold if the assumption (i) is relaxed. As a simple counterexample, let $d=2$, $j=1$, and $k=0$. Let $X$ be a polygon with a vertex at $o$ and let $\psi\equiv 1$. Then, $\Psi_{0,L}^\psi(X\cap L)=1$ for all $L\in\mathcal{L}_1^2$ and the left hand side of (16) becomes $c_{2,1}=\pi$. The inner integral at the right hand side of (16) is simply $|p(n|x)|$. Since $\{(o,n)\in\mathrm{nor} X\}$ has positive measure, the right hand side of (16) is undefined.
\end{rem}

In the special case $k=j-1$, $\Psi^\psi_{k,L} (X\cap L)$ is an integral with respect to the Hausdorff measure on the normal bundle of $X\cap L$. Since $\G\big(L^x, A_{\{1,\dots,d-1\}}\big)=\big|p(n|L^x)\big|$, Theorem \ref{rotational} shows that the rotational integral $\Psi(X)$ is again an integral with respect to the Hausdorff measure over the normal bundle of $X$. This is made precise by the following corollary.

\begin{cor}\label{rotcor}
Suppose $X\in \widetilde{\mathcal{PR}}^d$. Let $\psi: \mathcal{U}_j^d \to \R$ be a locally bounded measurable function and $1\leq  j<d $.
 Then
\begin{align*}
\int_{\Li_{j}^d} \Psi^\psi_{j-1,L} {}&(X\cap L)\, \di L\\
&= \int_{\nor X} \frac{1}{|x|^{d-j}} \int_{\Li_{j-1}^{x^\perp}} \psi\big(L^x,x,\pi\big(n | L^x\big)\big) \big|p(n | L^x)\big|\, \di L\, \Lambda_{d-1}(X;\di (x,n)) .
\end{align*}
\end{cor}

\subsection{Rotational Crofton formulae for Minkowski tensors}
If we choose 
\begin{equation*}
\psi(L,x,n) = \frac{\sigma_{j-k}}{r!s!\sigma_{j-k+s}} x^rn^s
\end{equation*}
in \eqref{psiE}, then $\Psi_{k,L}^\psi$ is the Minkowski tensor $\Phi^{r,s}_{k,L}$ in $L$ and Theorem \ref{rotational} becomes a result concerning the rotational integral of Minkowski tensors.

 The special case $s=0$ was treated in \cite[Proposition 5.3]{jeremy}.  For $s>0$ and $k=j-1$, the formula in Theorem \ref{rotational} can be given a more explicit expression. This is shown in the following theorem when $j>1$. The case $j=1$ is simpler and is postponed to Section \ref{j=0}.
To state the theorem, we introduce the following notation for $(x,n)\in \nor X$:
\begin{equation*}
\alpha= \alpha(x,n)=\sin(\angle(x,n)) = \sqrt{1-\tfrac{\langle x,n \rangle^2}{|x|^2}},
\end{equation*}
where $\angle(x,n)$ is the angle between $x$ and $n$. 

\begin{thm}\label{rotationmink}
Suppose $X\in \widetilde{\mathcal{PR}}^d$ and $1<j<d$. Then
\begin{align*}
 \int_{\Li_j^d} \Phi_{j-1,L}^{r,s}{}&(X \cap L)\, \di L = \frac{ \sigma_1 c_{d-3,j-2}}{r!s!\sigma_{s+1}} \sum_{a+b+c+2l=s} \begin{pmatrix}s\\a,b,c,2l\end{pmatrix}\frac{\sigma_{2l+d-2}\sigma_{d-1+2b+2c+4l}}{\sigma_{2l+1}\sigma_{j-1+2b+2c+2l}\sigma_{d-j+2l}} \\
&\times \sum_{p+q+t+v=l} \begin{pmatrix}l\\p,q,t,v\end{pmatrix} (-1)^{q+v+b}2^{t+1}    
 Q^p \int_{\nor X}  
n^{c+2q +t}   \frac{x^{r+a+b+2v+t}}{|x|^{d-j+a+b+2v+2t}}  \\
&\times{\alpha^{2p}\langle x,n \rangle^{a+b+t}} {}_2F_1\Big(\tfrac{s-1}{2},  \tfrac{d-j+2l}{2} ; \tfrac{d-1+2b+2c+4l}{2}    ; {\alpha^2}\Big) \,   \Lambda_{d-1}(X;\di (x,n)).
\end{align*}
For $\alpha=1$, the integrand should be interpreted as the limit when $\alpha \to 1$.
\end{thm}
We remark here that in the case where $X$ is convex and $o$ is an interior point of $X$, the situation $\alpha=1$ does not occur.

\begin{proof}
Corollary \ref{rotcor} shows that
\begin{align}\label{samlet}
 \int_{\Li_j^d} \Phi_{j-1,L}^{r,s}{}&(X \cap L)\ \di L \\
& =\frac{\sigma_1}{r!s!\sigma_{s+1}}\int_{\nor X} \frac{x^r}{|x|^{d-j}}  \int_{\Li_{j-1}^{x^\perp}} \frac{ p\big(n | L^x\big)^s }{\big|p\big(n | L^x\big)\big|^{s-1}}\, \di L\, \Lambda_{d-1}(X;\di (x,n)).\nonumber
\end{align}
We compute the inner integral. Write $n= n_x + n_{x^\perp}$ where $n_x=\langle n,x\rangle x/|x|^2$ is the  projection of $n$ onto $x$ and $n_{x^\perp}$ the projection of $n$ onto $x^\perp$. 
Then,
\begin{align*}
I:={}& \int_{\Li_{j-1}^{x^\perp}}  p\big(n|L^x\big)^s \big|p\big(n| L^x\big)\big|^{1-s}\, \di L = \int_{\Li_{j-1}^{x^\perp}}  \big(n_x + p\big(n_{x^\perp}| L\big)\big)^s \big|n_x + p\big(n_{x^\perp}| L\big)\big|^{1-s}\, \di L.
\end{align*}
If $n_{x^\perp}\neq 0$,  we may use
\eqref{coarea} and \eqref{sphereform} with $v=n_{x^\perp}/|n_{x^\perp}|$ and find
\begin{align}\nonumber
I{} 
&= c_{d-3,j-2}\int_{S^{d-2}(x^\perp )} \mathds{1}_{\{\langle u,n_{x^\perp}\rangle>0\}}  \bigg(\frac{\sqrt{|n_{x^\perp}|^2-\langle n_{x^\perp},u \rangle^2}}{\langle n_{x^\perp},u \rangle} \bigg)^{2-j}\frac{\big(n_x +\langle n_{x^\perp}, u\rangle u \big)^s}{ \big|n_x + \langle n_{x^\perp},u\rangle u \big|^{s-1}}\, \di u\\ \nonumber
&=  c_{d-3,j-2}\int_{S^{d-3}(x^\perp \cap n^\perp)}\int_{0}^1 { t }^{j-2}  (1-t^2)^{\frac{d-2-j}{2}} \frac{\big(n_x + \alpha t \big(t \alpha^{-1}{n_{x^\perp}} + \sqrt{1-t^2}\omega \big)\big)^s }{((1-\alpha^2) + \alpha^2 t^2  )^{\frac{s-1}{2}}}\,\di t\,\di \omega\\ \label{sumI}
&=  c_{d-3,j-2}\sum_{a+b+l=s}\begin{pmatrix}s\\a,b,l\end{pmatrix}n_x^a n_{x^\perp}^b\alpha^{l}\\
&\quad \times \int_{S^{d-3}(x^\perp \cap n^\perp)}\int_{0}^1  \omega^{l} t^{j-2+2b+l}((1-\alpha^2) + \alpha^2 t^2  )^{\frac{1-s}{2}} (1-t^2)^{\frac{d-2 -j+l}{2} }\, \di t\, \di \omega. \nonumber
\end{align}
Note that
\begin{equation}\label{Qint}
\int_{S^{d-3}(x^\perp \cap n^\perp)} \omega^{l}\, \di \omega = \begin{cases}
2\frac{\sigma_{l+d-2}}{\sigma_{l+1}} Q(x^\perp \cap  n^\perp)^{\frac{l}{2}}, & l  \text{ even},\\
0, & l \text{ odd},
\end{cases}
\end{equation}
as shown in e.g.\ \cite[(24)]{schsch}, and that 
\begin{align}\nonumber
 F_{d,j,s,l,b}(\alpha^2) {}&:=\int_{0}^1  t^{j-2+2b+2l}((1-\alpha^2) + \alpha^2 t^2  )^{\frac{1-s}{2}} (1-t^2)^{\frac{d-j-2}{2}+l}\, \di t\\ \nonumber
&=\frac{1}{2}\int_{0}^1  (1-t)^{\frac{j-3+2b+2l}{2}}(1 - \alpha^2 t  )^{\frac{1-s}{2}} t^{\frac{d-j-2+2l}{2}} \,\di t\\
&= \frac{\sigma_{d-1+2b+4l}}{\sigma_{j-1+2b+2l}\sigma_{d-j+2l}}{}_2F_1\Big(\tfrac{s-1}{2},  \tfrac{d-j+2l}{2} ; \tfrac{d-1+2b+4l}{2}    ; {\alpha^2}\Big) \label{hypergeo}
\end{align}
for $\alpha<1$. This yields
\begin{align}\label{longI}
I=c_{d-3,j-2} \sum_{a+b+2l=s} 2\frac{\sigma_{2l+d-2}}{\sigma_{2l+1}} \begin{pmatrix}s\\ a,b,2l\end{pmatrix}n_x^a n_{x^\perp}^b \alpha^{2l}  Q(x^\perp \cap  n^\perp)^l F_{d,j,s,l,b}(\alpha^2),
\end{align}
when $\alpha <1$.
For $\alpha = 1$, only terms with $a=0$ contribute to \eqref{sumI} since $n_x=0$. This corresponds to interpreting terms of the form
\begin{equation*}
n_x^a F_{d,j,s,l,b}(\alpha^2) 
\end{equation*}
in \eqref{longI} as the limit when $\alpha \to 1$. Indeed, this holds for $a=0$ because $s=b+2l$ in this case, and hence \eqref{value1} shows that $F_{d,j,s,l,b}(1)$ is finite. For $a>0$, we have
\begin{align*}
|n_x|^a {}& F_{d,j,s,l,b}(\alpha^2) = (1-\alpha^2)^{\frac{s-b-2l}{2}}\frac{1}{2}\int_{0}^1  (1-t)^{\frac{j-3+2b+2l}{2}}(1 - \alpha^2 t  )^{\frac{1-s}{2}} t^{\frac{d-j-2+2l}{2}}\, \di t\\
&=(1-\alpha^2)^{\frac{1}{2}}\frac{1}{2}\int_{0}^1  (1-t)^{\frac{j-3+2b+2l}{2}}\bigg(\frac{1-\alpha^2}{1-\alpha^2 t}\bigg)^{\frac{s-b-2l-1}{2}}(1 - \alpha^2 t  )^{\frac{-b-2l}{2}} t^{\frac{d-j-2+2l}{2}} \,\di t\\
&\leq (1-\alpha^2)^{\frac{1}{2}}\frac{1}{2} \int_{0}^1  (1-t)^{\frac{j-3+b}{2}} t^{\frac{d-j-2+2l}{2}}\, \di t. 
\end{align*}
Our assumptions on $j$ ensure that this converges to 0 when $\alpha\to 1$.

It is easy to check that the formula \eqref{longI} also holds when $n_{x^\perp}=0$ since $\alpha=0 $ in this case.
Finally, we use that $n_x=\langle n, x\rangle x/|x|^2$ and $n_{x^\perp}=n-\langle n, x\rangle x/|x|^2$ to obtain
\begin{align*}
n_x^a n_{x^\perp}^b 
&= \sum_{c=0}^b \binom{b}{c}(-1)^{b-c} n^c x^{a+b-c}  \langle n, x\rangle^{a+b-c} |x|^{-2(a+b-c)}
\end{align*}
and 
\begin{align} \label{expandQ}
Q(x^\perp \cap  n^\perp)^l {}&= \bigg( Q-\bigg(\frac{x}{|x|}\bigg)^2 - \bigg(\frac{n_{x^\perp}}{|n_{x^\perp}|}\bigg)^2 \bigg)^l \\
&=\sum_{p+q+v+t=l} \begin{pmatrix}l\\p,q,v,t\end{pmatrix} (-1)^{q+v}2^t\alpha^{-2(q+v+t)}\langle x,n  \rangle^{t}|x|^{-2v-2t}  Q^p n^{2q +t} x^{2v+t}. \nonumber
\end{align}
Inserting everything in \eqref{samlet} and renaming indices proves the theorem.
\end{proof}

\begin{ex}
Let $d=3$ and $j=2$. Previously, explicit formulae for
\begin{equation*}
\int_{\Li_2^3} \Phi_{2,L}^{r,0} (X\cap L)\, \di L \ \text{ and }\  \int_{\Li_2^3} \Phi_{1,L}^{r,0}(X\cap L)\, \di L 
\end{equation*}
have been given \cite[Example 5.2 and 5.4]{jeremy}. Theorem \ref{rotationmink} opens up for studying the integrals 
\begin{equation*}
\int_{\Li_2^3} \Phi_{1,L}^{r,s} (X\cap L)\, \di L
\end{equation*}
for arbitrary $s$.

 For $s=1$, we use that ${}_2F_1(0, b;c; \alpha^2)=1$ and get
\begin{align*}
 &\int_{\Li_2^3}  \Phi_{1,L}^{r,1}(X \cap L) \, \di L\\
 &=\frac{1 }{r!\pi} \sum_{a+b+c=1} \frac{\sigma_{2+2b+2c}}{\sigma_{1+2b+2c}} (-1)^{b}    
  \int_{\nor X} 
  n^{c}\frac{ x^{r+a+b}}{|x|^{d-j+a+b}}  {\langle x,n \rangle^{{a+b}}} \,  \Lambda_{2}(X;\di (x,n))\\
&= \frac{1 }{2r!} \bigg(\int_{\nor X} 
  \frac{x^{r+1}}{|x|^{2}} \langle x,n \rangle  \,  \Lambda_{2}(X;\di (x,n))+ \int_{\nor X} 
  n\frac{x^{r}}{|x|} \, \Lambda_{2}(X;\di (x,n))  \bigg).
\end{align*}
For $s=2$, Theorem \ref{rotationmink} yields the following expression
\begin{align}\label{3Ds2}
 &\int_{\Li_2^3} \Phi^{r,2}_{1,L}(X \cap L)\, \di L \\ \nonumber
&= \frac{  2}{\sigma_{3}r!}  \bigg( \sum_{a+b+c=2} \begin{pmatrix}2\\a,b,c\end{pmatrix}  (-1)^{b}      \int_{\nor X} 
n^{c}   x^{r+a+b}\frac{\langle x,n\rangle^{{a+b}}}{|x|^{1+a+b}}      F_{3,2,2,0,b+c}(\alpha^2)\, \Lambda_{2}(X;\di (x,n))\\ 
&+  \sum_{p+q+t+v=1} (-1)^{q+v}2^{t}     Q^p  \int_{\nor X} 
n^{2q +t}   x^{r+2v+t}\frac{\alpha^{2p}\langle x,n \rangle^{{t}}}{|x|^{1+2v+2t}}      F_{3,2,2,1,0}(\alpha^2) \, \Lambda_{2}(X;\di (x,n)) \bigg),\nonumber
\end{align}
where $ F_{d,j,s,l,b}$ is as in \eqref{hypergeo}. The hypergeometric functions involved can be found at \cite{wolfram}. 
If $K$ and $E$ denote the complete elliptic integrals of the first and second kind, respectively, we get
\begin{align*}
& F_{3,2,2,0,0}(\alpha^2)=\tfrac{\pi}{2}  {}_2F_1\Big(\tfrac{1}{2},  \tfrac{1}{2} , 1    ; {\alpha^2}\Big) = K(\alpha^2),\\
&F_{3,2,2,0,1}(\alpha^2)= \tfrac{\pi}{4}  {}_2F_1\Big(\tfrac{1}{2},  \tfrac{1}{2} , 2   ; {\alpha^2}\Big) 
= \alpha^{-2}(E(\alpha^2)+(\alpha^2-1)K(\alpha^2)),\\
&F_{3,2,2,1,0}(\alpha^2)=
\tfrac{\pi}{16}  {}_2F_1\Big(\tfrac{1}{2},  \tfrac{3}{2} , 3   ; {\alpha^2}\Big)  =\tfrac{1}{3\alpha^{4}} (2(\alpha^2-1)K(\alpha^2)- (\alpha^2-2)E(\alpha^2)),\\
&F_{3,2,2,0,2}(\alpha^2)=
\tfrac{3\pi}{16} {}_2F_1\Big(\tfrac{1}{2},  \tfrac{1}{2} , 3   ; {\alpha^2}\Big) 
= \tfrac{1}{3\alpha^{4}}((4\alpha^2-2)E(\alpha^2)+(3\alpha^4 -5\alpha^2 + 2)K(\alpha^2) ).
\end{align*}
This can be inserted in \eqref{3Ds2} to simplify the expression, but the functions $E(\alpha^2)$ and $K(\alpha^2)$ do not cancel out.
\end{ex}

\subsection{The case $j=1$}\label{j=0}
If $L\in {\Li_{0}^{x^\perp}}$, then $L^x$ is the line spanned by $x$. Moreover, if $x$ and $n$ are non-orthogonal, then $\pi(n| L^x)=\frac{\langle x,n \rangle x}{|\langle x,n \rangle ||x|}$.  Thus, Corollary \ref{rotcor} becomes 
\begin{align*}
\int_{\Li_1^d} \Psi^\psi_{0,L}(X \cap L) \,\di L & = \int_{\nor X} \frac{1}{|x|^{d-1}}  \int_{\Li_{0}^{x^\perp}}  \psi\big(L^x,x,\pi\big(n|L^x\big) \big) \big|p(n | L^x)\big|\, \di L\, \Lambda_{d-1}(X;\di (x,n))\\ 
 &=\frac{1}{\sigma_{1}}\int_{\nor X} \frac{\psi\Big(\text{span}(x),x, \frac{\langle x,n \rangle x}{|\langle x,n \rangle ||x|}  \Big)}{|x|^{d}} |\langle x, n \rangle|  \, \Ha^{d-1}(\di x).
\end{align*}
In the special case of Minkowski tensors, this yields 
\begin{align}\label{j1mink}
\int_{\Li_1^d} \Phi_{0,L}^{r,s}(X \cap L)\, \di L
& =\frac{1}{r!s!\sigma_{s+1}}\int_{\nor X} \frac{x^{r+s}\langle x,n \rangle^s}{|x|^{d+s}|\langle x,n \rangle|^{s-1}}  \, \Ha^{d-1}(\di x).
\end{align}
We remark that if $X\in \mathcal{K}^d$ and $o$ lies in the interior of $X$, then $\pi(n| L^x)={x}/{|x|}$, so \eqref{j1mink} simplifies to
\begin{equation*}
\int_{\Li_1^d} \Phi_{0,L}^{r,s}(X \cap L)\, \di L
=\frac{1}{r!s!\sigma_{s+1}}\int_{\nor X} \frac{x^{r+s}}{|x|^{d+s-1}}\langle x,n\rangle \,\Ha^{d-1}(\di x).
\end{equation*}

\subsection{The case $j=d-1$}
In the case $j=d-1$, the rotational integral in Theorem \ref{rotational} can also be computed explicitly. We demonstrate this only for Minkowski tensors $\Phi_{k,L}^{r,s}$ with $k< d-2$ since the case $j=d-1$ and $k=d-2$ is covered by Theorem \ref{rotationmink}. We get 

\begin{align} \nonumber 
\int_{\Li_{d-1}^d}{}& \Phi^{r,s}_{k,L}(X \cap L) \,\di L =\frac{1}{2r!s!\sigma_{d-1-k+s}}\int_{\nor X} \frac{x^r}{|x|}\sum_{|I|=d-2-k} \frac{\prod_{i\in I} \kappa_i(x,n)}{\prod_{i=1}^{d-1} \sqrt{1+\kappa_i(x,n)^2}} \\ \nonumber
&\times \int_{S^{d-2}({x^\perp})}(n-\langle n,u \rangle u)^s  |p(u|A_I(x,n))|^2 { |n-\langle n, u\rangle u |^{k-d+1-s}}\, \di u \, \Ha^{d-1}(\di (x,n))\\ \nonumber
={}& \frac{1}{2r!s!\sigma_{d-1-k+s}}\int_{\nor X} \sum_{a+b=s}\binom{s}{a} (-1)^b  n^a\frac{x^r}{|x|}\sum_{|I|=d-2-k} \frac{\prod_{i\in I} \kappa_i(x,n)}{\prod_{i=1}^{d-1} \sqrt{1+\kappa_i(x,n)^2}} \\ \label{d-1int} 
&\times \sum_{i\notin I} \text{Contr} \bigg(\int_{S^{d-2}({x^\perp})}\langle n,u \rangle^b u^{b+2}  { (1-\langle n, u\rangle^2)^{\frac{k-d+1-s}{2}}}\, \di u , a_i^2\bigg)\, \Ha^{d-1}(\di (x,n)).
\end{align}
If $x$ and $n$ are not parallel, then \eqref{sphereform} with $v=\pi(n|x^\perp)$ yields
\begin{align*}
{}&\int_{S^{d-2}({x^\perp})}\langle n,u \rangle^b u^{b+2}  { (1-\langle n, u\rangle^2 )^{\frac{k-d+1-s}{2}}} \,\di u \\
&\quad =\alpha^b  \int_{S^{d-3}({x^\perp}\cap n^\perp)} \int_{-1}^1t^b(1-t^2)^{\frac{d-4}{2}}\big(t\pi(n|x^\perp)+\sqrt{1-t^2}w\big)^{b+2}  { (1-t^2\alpha^2 )^{\frac{k-d+1-s}{2}}}\, \di t\, \di w\\
&\quad =2\alpha^b \sum_{2p+q=b+2} \binom{b+2}{2p} \frac{\sigma_{2p+d-2}}{\sigma_{2p+1}}\frac{\Gamma\big(\frac{b+q+1}{2}\big)\Gamma\big( \frac{2p +d-2}{2}\big)}{\Gamma\big( \frac{2b+d+1 }{2}\big)}  \pi(n|x^\perp)^q  Q({x^\perp}\cap n^\perp)^p\\
 &\quad \quad \times {}_2F_{1}\Big(\tfrac{d-1-k+s}{2}, \tfrac{b+q+1}{2}; \tfrac{2b+1+d}{2};\alpha^2\Big).
\end{align*}
If $x$ and $n$ are parallel, the same holds when $\pi(n|x^\perp)$ is interpreted as any vector $v\in S^{d-2}(x^\perp)$ and $Q(n^\perp \cap x^\perp)$ as $Q(x^\perp \cap v^\perp)$. 


When $\alpha\neq 0$, we may compute
\begin{align}\nonumber
\binom{b+2}{2}{}&\text{Contr} \Big(\pi(n|x^\perp)^q  Q({x^\perp}\cap n^\perp)^p , a_i^2\Big)\\ \nonumber
 ={}&   
  \binom{q}{2}  \pi(n|x^\perp)^{q-2}  Q({x^\perp}\cap n^\perp)^p \text{Contr}\Big(\pi(n|x^\perp), a_i\Big)^2\\ \nonumber
  &+
	2pq \, \pi(n|x^\perp)^{q-1} Q({x^\perp}\cap n^\perp)^{p-1}  \text{Contr}\Big(\pi(n|x^\perp), a_i\Big) \text{Contr}\Big(Q({x^\perp}\cap n^\perp), a_i\Big) \\ \nonumber
	&+
4\binom{p}{2} 
 \pi(n|x^\perp)^q  Q({x^\perp}\cap n^\perp)^{p-2} \text{Contr}\Big(Q({x^\perp}\cap n^\perp), a_i\Big)^2 \\
&+p\, \pi(n|x^\perp)^q  Q({x^\perp}\cap n^\perp)^{p-1} \text{Contr}\Big(Q({x^\perp}\cap n^\perp), a_i^2\Big) , \label{longcontr}
\end{align}
where negative powers of a tensor are interpreted as zero and
\begin{align*}
{}&\text{Contr}\Big(\pi(n|x^\perp), a_i\Big) = \langle \pi(n|x^\perp), a_i \rangle =  \frac{\langle x,n\rangle \langle x,a_i \rangle}{\alpha|x|^2},\\
{}& \text{Contr}\Big(Q({x^\perp}\cap n^\perp), a_i\Big)= p(a_i| x^\perp \cap n^\perp ) = a_i - \frac{\langle x,a_i \rangle}{\alpha|x|}\pi(x|n^\perp) ,\\
{}& \text{Contr}\Big(Q({x^\perp}\cap n^\perp), a_i^2\Big) = | p(a_i| x^\perp \cap n^\perp)|^2 =  1 - \frac{\langle x,a_i \rangle^2}{\alpha^2|x|^2}.
\end{align*}
This can be inserted in \eqref{d-1int} to provide a formula for the rotational integral.

\begin{ex}  
In dimension $d=3$, the simplest example with $j=d-1=2$ and $k<j-1=1$ is $\Phi_{0,L}^{r,0}$. To the best of our knowledge, this situation has not been treated in the literature.
Using the above computations in this case, we get
\begin{align*}
\text{Contr}{}&\bigg(\int_{S^{1}({x^\perp})}  \frac{u^{2}}  { 1-\langle n, u\rangle^2 } \, \di u ,a_i^2\bigg)\\
&={\pi } \text{Contr}\Big(    Q({x^\perp}\cap n^\perp) {}_2F_{1}\Big(1, \tfrac{1}{2}; 2;\alpha^2\Big) +  \pi(n|x^\perp)^2  {}_2F_{1}\Big(1, \tfrac{3}{2}; 2;\alpha^2\Big) ,a_i^2\Big)\\
&={\pi }  \bigg( \frac{2-2\sqrt{1-\alpha^2}}{\alpha^2} - 2\frac{(\sqrt{1-\alpha^2}-1)^2}{\alpha^4}\frac{\langle x, a_i\rangle^2}{|x|^2}      \bigg).
\end{align*}
This should be interpreted as $\frac{\pi}{2}$ when $\alpha=0$.  The values of the hypergeometric functions are taken from \cite{wolfram}. Inserting in \eqref{d-1int}, we get
\begin{align*}
\int_{\Li_{2}^3} \Phi_{0,L}^{r,0}(X \cap L)\, \di L 
= {}&\frac{1}{2r!}  \int_{\nor X} \frac{x^r}{|x|}\sum_{i=1}^2 \frac{ \kappa_i(x,n)}{\prod_{j=1}^{2} \sqrt{1+\kappa_j(x,n)^2}} \\ 
&\times  \bigg( \frac{1-\sqrt{1-\alpha^2}}{\alpha^2} - \frac{(\sqrt{1-\alpha^2}-1)^2}{\alpha^4}\frac{\langle x,a_{3-i}\rangle^2}{|x|^2}      \bigg)\, \Ha^{2}(\di (x,n)).
\end{align*}
\end{ex}

\section{Affine Crofton formulae}\label{affinesec}

\subsection{General affine formulae}\label{affsubsec}
In this section, we consider  for each affine subspace $E\in \mathcal{E}_j^d$ a valuation $\Psi^\psi_{k,E}$ defined on  compact  sets of positive reach $X\subseteq E$ by 
\begin{equation*}
\Psi^\psi_{k,E}(X)=\int_{\Sigma^E} \psi(E,x,n) \,\Lambda^{E}_k(X,\di(x,n)),
\end{equation*}
$0\leq k<j<d$.
Here $\psi: \mathcal{U}_j^d \to \R $ is a locally bounded measurable function, where $\mathcal{U}_j^d$ is as in \eqref{Uaff}. 

Suppose $X\subseteq \R^d$ is a compact set of positive reach. It follows from \cite[Theorem 6.11 (1)]{federer2} that for almost all $E\in \mathcal{E}_j^d$, 
the set $X\cap E$ has positive reach and hence $\Psi^\psi_{k,E}(X\cap E)$ is well-defined. The integral of $\Psi^\psi_{k,E}(X\cap E)$ with respect to the motion invariant measure on $\mathcal{E}_j^d$ is determined in the next theorem.
\begin{thm}\label{crofton}
Let $X\in \mathcal{PR}^d$ and $0\leq k<j<d$. Then,
\begin{align}\label{croftonresult}
\int_{\mathcal{E}_j^d} \Psi^\psi_{k,E}(X\cap E) \, \di E ={}& \frac{1}{ \sigma_{j-k}} \int_{\nor X} \sum_{|I|=j-k-1}   \frac{\prod_{i\in I} \kappa_i(x,n)}{\prod_i \sqrt{1+\kappa_i(x,n)^2}}\\
& \times  \int_{\mathcal{\Li}_{j}^d}   \psi(L+x,x,\pi(n|L)) \frac{ \mathcal{G}(L, A_I(x,n) )^2}{|p(n|L)|^{j-k}} \, \di L \, \Ha^{d-1} (\di (x,n)). \nonumber
\end{align}
For $k=j-1$, this can be simplified to
\begin{align*}
\int_{\mathcal{E}_j^d} \Psi^\psi_{j-1,E}(X\cap E) \, \di E {}&=  \int_{\nor X} \int_{\mathcal{\Li}_{j}^d} \psi(L+x,x,\pi(n|L)) {|p(n|L)|} \,   \di L \, \Lambda_{d-1} (X;\di (x,n)).
\end{align*}
\end{thm}
\begin{proof}
It follows from \cite[Theorem 3.1]{rataj} that
\begin{align}\nonumber
\int_{\mathcal{E}_j^d} \Psi^\psi_{k,E}(X\cap E) \, \di E ={}& \frac{1}{ \sigma_{j-k}} \int_{\mathcal{\Li}_{j}^d}  \int_{\nor X} \sum_{|I|=j-k-1}   \frac{\prod_{i\in I} \kappa_i(x,n)}{\prod_i \sqrt{1+\kappa_i(x,n)^2}}\\
& \times    \psi(L+x,x,\pi(n|L)) \frac{ \mathcal{G}(L, A_I(x,n) )^2}{|p(n|L)|^{j-k}}\,   \Ha^{d-1} (\di (x,n)) \, \di L, \label{fubini}
\end{align}
 since the condition \cite[(3.1)]{rataj} is satisfied for almost all $L\in \Li_j^d$, as noted in the proof of \cite[Theorem 3.5]{rataj}. One can show, using an argument similar to the one in the proof of Theorem \ref{rotational} given in the appendix, that Fubini's theorem can be applied to \eqref{fubini}. This yields \eqref{croftonresult}. The last statement follows because $\mathcal{G}(L, A_{\emptyset}(x,n) ) = |p(n|L)|$.
\end{proof}

In the case where $\psi(E,x,n)$ does not depend on $E$, the following theorem is a direct consequence of \cite[Theorem 3.5]{rataj}. To state the result, we introduce the constant 
\begin{align*}
C_{d,j,k}={}&c_{d,j} \binom{d+k-j-1}{k}\frac{\Gamma\big(\frac{j+1}{2}\big)\Gamma\big(\frac{d-j+1}{2}\big)}{ \pi^{\frac{d}{2}}},\quad 0\leq k<j<d,
\end{align*}
and, given $n\in S^{d-1}$, we let $S^{d-1}_+(n)=\big\{z\in S^{d-1} \mid \langle z,n \rangle \geq  0 \big\}$ denote the upper halfsphere determined by $n$.

\begin{thm}\label{denne}
Let $X\in \mathcal{PR}^d$ and $0\leq k<j<d$. Suppose $\psi: \R^d \times S^{d-1} \to  \R$ is a measurable, locally bounded function. Then,
\begin{align} \label{factorsigma}
\int_{\mathcal{E}_j^d} \Psi^\psi_{k,E}(X{}&\cap E) \, \di E 
=
\frac{C_{d,j,k} }{\sigma_{j-k}} \int_{\nor X}      \sum_{l=1}^{d-1} \sum_{\substack{|I|= j-k-1,\\ l\notin I}} \frac{\prod_{i\in I }\kappa_i(x,n) }{\prod_i \sqrt{1+\kappa_i(x,n)^2}} \\
&\times  \int_{S^{d-1}_+(n)} \psi(x,z) (1-\langle z,n\rangle^2)^{-\frac{j+1}{2}}\langle z,n\rangle^{k+1}  \langle z,a_l \rangle^2 \, \di z \,
\Ha^{d-1} (\di (x,n)).\nonumber
\end{align}
\end{thm}
 Note that the factor $1/\sigma_{j-k}$ in \eqref{factorsigma} also appears in the proof of \cite[Theorem 3.5]{rataj}, but seems to be forgotten in the statement of the theorem.

The approach in \cite{schuster}, using the explicit expression for the curvature measures for polytopes, also relies on the result in Theorem \ref{denne} in the special case where $X$ is a polytope. 
%
  
If  $\psi(E,x,n) $ does not depend on $E$ and $n$, then \eqref{factorsigma} becomes particularly nice.
\begin{cor}\label{psixcor}
Let $X\in \mathcal{PR}^d$ and $0\leq k<j<d$. Suppose $\psi(E,x,n) =\psi(x)$ is a locally bounded measurable function. Then
\begin{align*}
\int_{\mathcal{E}_j^d} \Psi^\psi_{k,E}(X\cap E)\, \di E {}&=c_{d,j}\frac{\Gamma\big(\frac{j+1}{2}\big)\Gamma\big(\frac{d+k-j+1}{2}\big)}{\Gamma\big(\frac{k+1}{2}\big)\Gamma\big(\frac{d+1}{2}\big)} \Psi_{d-j+k,\R^d}^{\psi} (X). 
\end{align*}
\end{cor}

\begin{proof}
We find, using \eqref{sphereform}  with $v=n$, that 
\begin{align*}
\int_{S^{d-1}_+(n)} (1-\langle z,n\rangle^2)^{-\frac{j+1}{2}}\langle z,n\rangle^{k+1}  \langle z,a_l \rangle^2 \, \di z{}&=\int_0^1 t^{k+1}(1-t^2)^{\frac{d-j-2}{2}}\,\di t  \int_{S^{d-2}(n^\perp)} \langle w,a_l \rangle^2 \di w\\
&=\frac{\Gamma\big(\frac{k+2}{2}\big)\Gamma\big(\frac{d-j}{2}\big)}{2\Gamma\big(\frac{d-j+k+2}{2}\big)} \frac{\sigma_{d-1}}{(d-1)}.
\end{align*}
The result now follows from Theorem \ref{denne}.
\end{proof}
Note that for $\psi=1$, Corollary \ref{psixcor} reduces to the classical Crofton formula.

\subsection{Affine Crofton formulae for Minkowski tensors}
By choosing 
\begin{equation*}
\psi(E,x,n) = \frac{\sigma_{j-k}}{r!s!\sigma_{j-k+s}} x^rn^s
\end{equation*}
in Theorem \ref{denne}, we obtain affine Crofton formulae for Minkowski tensors. Such formulae were first given in \cite[Theorem 2.5 and 2.6]{schuster} in the case of convex sets. These theorems show that the integral of the Minkowski tensors $\Phi_{k,E}^{r,s}(K\cap E)$ with respect to the motion invariant measure on $\mathcal{E}_j^d$ is again a linear combination of Minkowski tensors as one would expect from Alesker's classification theorem mentioned in Section \ref{minkowskisec}.  However, the constants appearing in the linear combinations are complicated to evaluate. Recently, the results have been generalized and the constants have been simplified in \cite{weis}. It is possible to derive the constants in Theorem \ref{ksmall} below directly from the formulae in \cite{schuster} using the identity \eqref{value1} or from \cite[Theorem 2]{weis} by rearranging terms. The main contribution of our Theorem \ref{ksmall} is the generalization of the formulae to sets of positive reach. 
Like the results in \cite{schuster} and \cite{weis}, the proof of Theorem \ref{ksmall} relies on Theorem \ref{denne} shown in \cite{rataj}.

Theorem \ref{ksmall} is stated, using the tensors $\Phi_{k}^{r,s,1}(X)$ that were introduced in Section \ref{minkowskisec}. These tensors can be written as a linear combination of the Minkowski tensors according to Proposition \ref{linkombprop}. 

\begin{thm}\label{ksmall}
Let $X\in \mathcal{PR}^d$ and $0\leq k < j<d$. Then 
\begin{align*}
\int_{\mathcal{E}_j^d}{}&\Phi_{k,E}^{r,s}(X\cap E) \, \di E {}
={} \frac{C_{d,j,k} \pi^{\frac{d-1}{2}}}{2\sigma_{j-k+s}\Gamma\big(\frac{d-j+2+k+s}{2}\big)} \\
& \times
 \sum_{p=0}^{\lfloor \frac{s}{2} \rfloor} \chi^p_{d,j,k,s}  \Big((d-j+k)Q^{p}\Phi_{d-j+k}^{r,s-2p}(X) + 2pQ^{p-1}\Phi_{d-j+k}^{r,s-2p,1}(X)\Big),
\end{align*}
where the constants $\chi^p_{d,j,k,s}$ are given in \eqref{chi} below. 
\end{thm}
Using \eqref{r0}, we get the following corollary for $r=0$, which was proven for convex sets in \cite[Corollary 1]{weis}.
\begin{cor}
Let $X\in \mathcal{PR}^d$ and $0\leq k < j <d$. Then
\begin{align*}
\int_{\mathcal{E}_j^d}{}& \Phi_{k,E}^{0,s}(X\cap E)\, \di E 
={} \frac{C_{d,j,k} \pi^{\frac{d-1}{2}}}{2\sigma_{j-k+s}\Gamma\big(\frac{d-j+2+k+s}{2}\big)} \\
&\times \sum_{p=0}^{\lfloor \frac{s}{2} \rfloor} \Big(  (d-j+k+2p)\chi^p_{d,j,k,s} -4\pi( p+1)(s-2p) \chi_{d,j,k,s}^{p+1}\Big)Q^{p}\Phi_{d-j+k}^{0,s-2p}(X),
\end{align*}
where the constants $\chi^p_{d,j,k,s}$ are given in \eqref{chi} below ($\chi^p_{d,j,k,s}=0$ if $p>\frac{s}{2}$).
\end{cor}

\begin{proof}[Proof of Theorem \ref{ksmall}]
Using Theorem \ref{denne} with
\begin{equation*}
\psi(x,n)=\frac{\sigma_{j-k}}{r!s!\sigma_{j-k+s}} x^rn^s,
\end{equation*}
we find 
\begin{align*}
\int_{\mathcal{E}_j^d}{}\Phi_{k,E}^{r,s}(X\cap E) \, \di E  
={}& \frac{C_{d,j,k}}{r!s!\sigma_{j-k+s}} \int_{\nor X}  x^r\sum_{l=1}^{d-1} \sum_{\substack{|I|= j-k-1,\\ l\notin I}}  \frac{\prod_{i\in I }\kappa_i }{\prod_i \sqrt{1+\kappa_i^2}}\\
&\times \int_{S^{d-1}_+(n)} z^s (1-\langle z,n \rangle^2)^{-\frac{j+1}{2}}\langle z,n \rangle^{k+1}   \langle z,a_l\rangle^2\,  \di z \, \Ha^{d-1}(\di(x,n)).
\end{align*}
We now use that 
\begin{align*}
\int_{S^{d-1}_+(n)} {}& z^s (1-\langle z,n \rangle^2)^{-\frac{j+1}{2}}\langle z,n \rangle^{k+1}   \langle z,a_l\rangle^2\,  \di z \\&= \textrm{Contr} \Big(\int_{S^{d-1}_+(n)} z^{s+2} (1-\langle z,n \rangle^2)^{-\frac{j+1}{2}}\langle z,n \rangle^{k+1}\,  \di z,a_i^2\Big).
\end{align*}
Applying \eqref{sphereform} with $v=n$, we obtain
\begin{align*}
\int_{S^{d-1}_+(n)}{}& z^{s+2} (1-\langle z,n \rangle^2)^{-\frac{j+1}{2}}\langle z,n \rangle^{k+1}\,  \di z \\
&=  \int_{S^{d-2}(n^\perp)}\int_0^1 (tn+\sqrt{1-t^2}w)^{s+2} (1-t^2)^{\frac{d-j-2}{2}}t^{k+1}\, \di t \, \di w\\
&={}\sum_{a+2b=s+2} \binom{s+2}{a} n^a \int_{S^{d-2}(n^\perp)}w^{2b}\, \di w \int_0^1  (1-t^2)^{\frac{2b+d-j-2}{2}}t^{a+k+1} \, \di t \\
&={}\sum_{a+2b=s+2} \sum_{p+q=b} (-1)^q \binom{b}{p} \binom{s+2}{a} \frac{\sigma_{d+2b-1}}{\sigma_{2b+1}}  \frac{\Gamma\big(\frac{k+2+a}{2}\big)\Gamma\big( \frac{d-j-2+2b}{2}\big)}{\Gamma\big( \frac{d+k-j+s+2}{2}\big)} Q^p  n^{a+2q},
 \end{align*}
where we have used that
\begin{equation*}
\int_{S^{d-2}(n^\perp)} w^l dw = \begin{cases}
2\frac{\sigma_{l+d-1}}{\sigma_{l+1}} Q(n^\perp)^{\frac{l}{2}}, & l \text{ even,}\\
0, & l \text{ odd},
\end{cases}
\end{equation*}
and that $Q(n^\perp) = Q - n^2$.
Since
\begin{equation*}
\text{Contr}\big( Q^p  n^{a+2q}, a_l^2 \big) = \binom{s+2}{2}^{-1} p\big( Q^{p-1}  n^{a+2q} + 2(p-1) Q^{p-2} n^{a+2q} a_l^2\big),
\end{equation*}
we get
\begin{align*}
\int_{\mathcal{E}_j^d} {}& \Phi_{k,E}^{r,s}(X\cap E) \,\di E {}
=
 \frac{C_{d,j,k} }{r!s!\sigma_{j-k+s}}  \sum_{b=0}^{\lfloor \frac{s+2}{2}\rfloor} \sum_{p=0}^b (-1)^{b-p} \binom{b}{p} \binom{s+2}{2b} \frac{\sigma_{2b+d-1}}{\sigma_{2b+1}}\\
&\times \frac{\Gamma\big(\frac{k+4+s-2b}{2}\big)\Gamma\big( \frac{d-j-2+2b}{2}\big)}{\Gamma\big( \frac{d+k-j+s+2}{2}\big)}  \int_{\nor X}   x^r  \sum_{l=1}^{d-1}  \binom{s+2}{2}^{-1} p \\
&\times\Big( Q^{p-1}  n^{s+2 - 2p} + 2(p-1) Q^{p-2} n^{s+2-2p} a_l^2\Big) \sum_{|I|= j-k-1, l\notin I}  \frac{\prod_{i\in I }\kappa_i }{\prod_i \sqrt{1+\kappa_i^2}} \,
\Ha^{d-1} (\di (x,n)) \\
 ={}&\frac{C_{d,j,k} \pi^{\frac{d-1}{2}}}{2\sigma_{j-k+s}\Gamma\big(\frac{d-j+2+k+s}{2}\big)}  \sum_{p=0}^{\lfloor \frac{s}{2} \rfloor} \chi^p_{d,j,k,s}  \Big((d-j+k)Q^{p}\Phi_{d-j+k}^{r,s-2p}(X) + 2pQ^{p-1}\Phi_{d-j+k}^{r,s-2p,1}(X)\Big),
\end{align*}
where
\begin{align}\label{chi}
\chi^p_{d,j,k,s} {}& 
= \frac{\sigma_{j-k+s-2p} }{2^{2p}p!\pi^{1/2}} \sum_{b=0}^{\lfloor \frac{s}{2}\rfloor-p} (-1)^{b}\binom{s-2p}{2b}  \frac{\Gamma\big(\frac{2b+1}{2}\big)\Gamma\big(\frac{k+2+s-2b-2p}{2}\big)\Gamma\big(\frac{d-j+2b+2p}{2}\big)}{ \Gamma\big(\frac{2b+2p+d+1}{2}\big)} \\
\nonumber
&=  
\frac{\sigma_{j-k+s-2p}\Gamma\big(\frac{k+2+s-2p}{2}\big)\Gamma\big(\frac{d-j+2p}{2}\big) }{p!2^{2p}\Gamma\big(\frac{2p+d+1}{2}\big)  }\sum_{b=0}^{ \lfloor \frac{s}{2}\rfloor -p} \frac{\big(-\frac{s-2p}{2}\big)_b \big(-\frac{s-2p-1}{2}\big)_b\big(\frac{d-j+2p}{2}\big)_b}{(1)_b  \big(\frac{2p+d+1}{2}\big)_b \big(-\frac{k+s-2p}{2}\big)_b} \\
\nonumber
&=\frac{\sigma_{j-k+s-2p}\Gamma\big(\frac{d-j+2p}{2}\big)\Gamma\big( \frac{k+2+s-2p}{2}\big)}{p!2^{2p}\Gamma\big(\frac{2p+d+1}{2}\big)} {}_3F_2\Big(\tfrac{2p-s+1}{2}, \tfrac{2p-s}{2}, \tfrac{d-j+2p}{2}; \tfrac{2p-k-s}{2},\tfrac{2p+d+1}{2};1\Big).
\end{align}
\end{proof}

In the special case $k=j-1$, we obtain the following simplification, see also \cite[Corollary 5]{weis}.
\begin{cor}\label{surfacedirect}
Let $X\in \mathcal{PR}^d$ and let $1\leq j<d$ and $k=j-1$. Then,
\begin{align*}
\int_{\mathcal{E}_j^d}  \Phi_{j-1,E}^{r,s}(X\cap E) \,\di E ={}& \frac{c_{d-2,j-1}\pi^{\frac{d+1}{2}}}{ \sigma_{s+1}\Gamma\big(\frac{d+s+1}{2}\big) } \sum_{p=0}^{\lfloor \frac{s}{2}\rfloor} \frac{\chi^p_{d-2,j-2,j-1,s}}{s-2p-1}  Q^p \Phi_{d-1}^{r,s-2p}(X),
\end{align*}
where $\chi^p_{d-2,j-2,j-1,s}$ is given in \eqref{chi} (we interpret ${2\pi\sigma_m}/{m} $ as $\sigma_{m+2}$ if $m= 0,-1$).
\end{cor}

\begin{proof}
This follows either directly from Theorem \ref{denne} using a computation of
\begin{equation*}
 \int_{S^{d-1}_+(n)} z^s (1-\langle z,n\rangle^2)^{-\frac{j-1}{2}}\langle z,n\rangle^{k+1} \,  \di z,
\end{equation*} 
 or from Theorem \ref{ksmall} using the identity \cite{wolfram32} 
\begin{equation*}
d(d+1)e \Big({}_3F_2(a,b,c;d,e;z) - {}_3F_2(a,b,c;d+1,e;z)\Big)=abcz{}_3F_2(a+1,b+1,c+1;d+2,e+1;z).
\end{equation*} 
\end{proof}

\subsection{Crofton integrals for spherical harmonics}
In the case $k=j-1$, the Crofton integral has a particularly nice expression in terms of spherical harmonics. This is analogous to \cite[Corollary 6.1]{bernighug}. 
\begin{cor}\label{harmonic}
Let $X\in \mathcal{PR}^d$ where $d\geq 3$ and let $1\leq j<d$ and $k=j-1$. Assume $\psi(E,x,n)=f(x)h(n)$ where $h $ is a $d$-dimensional spherical harmonic of degree $s$. Then,
\begin{align*}
\int_{\mathcal{E}_j^d} \Psi^\psi_{j-1,E}(X\cap E)\, \di E {}&= {c_{d-2,j-1}\sigma_{d-1}}a_{s,j,d} \Psi_{d-1,\R^d}^\psi(X),
\end{align*}
where
\begin{align*}
\sigma_{d-1} a_{2m,j,d}{}&=  (-1)^m \frac{\pi^{\frac{d-2}{2}}\Gamma\big(\frac{2m+1}{2}\big) \Gamma\big(\frac{d-j}{2}\big) \Gamma\big(\frac{j+1}{2}\big) }{\Gamma\big(\frac{2m+d-1}{2}\big)\Gamma\big(\frac{d+1}{2}\big) }  {}_3F_2\Big(-m,\tfrac{2m+d-2}{2},\tfrac{j+1}{2};\tfrac{1}{2},\tfrac{d+1}{2};1\Big),\\
\sigma_{d-1} a_{2m+1,j,d}{}&=  (-1)^m \frac{2\pi^{\frac{d-2}{2}} \Gamma\big(\frac{2m+3}{2}\big) \Gamma\big(\frac{d-j}{2}\big) \Gamma\big(\frac{j+2}{2}\big) }{ \Gamma\big(\frac{2m+d-1}{2}\big) \Gamma\big(\frac{d+2}{2}\big) }  {}_3F_2\Big(-m,\tfrac{2m+d}{2} ,\tfrac{j+2}{2};\tfrac{3}{2},\tfrac{d+2}{2};1\Big)  .
\end{align*}
\end{cor}

\begin{proof}
Applying  \eqref{sphereform} and \cite[Theorem 3.4.1]{groemer}, we get
\begin{align*}
\int_{S^{d-1}_+(n)} \langle z,n \rangle^{j}  (1-\langle z,n \rangle^2)^{\frac{-j+1}{2}} h(z)\, \di z=\sigma_{d-1} h(n) \int_0^1   (1-t^2)^{\frac{d-j-2}{2}} t^{j} P_s^d(t) \, \di t,
\end{align*}
where $P_s^d(t)$ is the Legendre polynomial of dimension $d$ and degree $s$ (with the notation of \cite{groemer}). The constants
\begin{align*}
a_{s,j,d} = \int_0^1   (1-t^2)^{\frac{d-j-2}{2}} t^{j} P_s^d(t)\, \di t
\end{align*}
are computed explicitly in \cite[Proposition 3]{mk} (with the same notation). 
A more compact expression can be found by consulting the integral table \cite{grad}.
The polynomial called $\binom{s+d-3}{d-3}P_s^d$ in \cite{groemer} is here referred to as the Gegenbauer polynomial $C_s^{(d-2)/2}$. We perform a substitution 
\begin{align*}
 \int_0^1   (1-t^2)^{\frac{d-j-2}{2}} t^{j} P_s^d(t)\, \di t=
 \frac{1}{2} \binom{s+d-3}{d-3}^{-1}\int_0^1   (1-y)^{\frac{d-j-2}{2}} y^{\frac{j-1}{2}} C_s^{\frac{d-2}{2}}\big(y^{\frac{1}{2}}\big)\, \di y.
\end{align*}
The latter integral can be found in \cite[7.319]{grad}. Theorem \ref{denne} then yields the result.
\end{proof}

Corollary \ref{harmonic} has a nice application to Minkowski tensors.
It was shown in \cite[Proposition 4.16]{bernighug} that the following tensor of rank $s$
\begin{equation*}
H^s_d (u) = \sum_{i=0}^{\lfloor \frac{s}{2}\rfloor} \frac{(-1)^i \Gamma\big(\frac{d}{2}+s-1-i\big)}{4^i i!(s-2i)!}|u|^{2i}  Q^i u^{s-2i}
\end{equation*}
has the property that all its coordinates are $d$-dimensional spherical harmonics. Replacing $u^s$ by $H_d^s(u)$ in the definition of the Minkowski tensors, we get the harmonic Minkowski tensors 
\begin{align*}
\Xi_{k}^{r,s}(X) = \frac{\sigma_{d-k}}{r!s!\sigma_{d-k+s}} \int_{\Sigma} x^r H_d^s(n)\, \Lambda_k(X; \, \di (x,n)).
\end{align*}

Similarly, for $E\in \mathcal{E}_j^d$ and $X \subseteq E$ we may define  
\begin{equation*}
\tilde{\Xi}_{j-1,E}^{r,s}(X)=\frac{\sigma_1}{r!s!\sigma_{s+1}}\int_{\Sigma^E} x^r H_d^s(n) \,\Lambda_{j-1}^E(X;\di (x,n)).
\end{equation*}
The tilde in $\tilde{\Xi}_{j-1,E}^{r,s}(X)$ indicates that this is generally not a harmonic tensor when $j=\dim E < d$ since $H_d^s(u)$ restricted to $S^{j-1}( E)$ is not necessarily harmonic. 
\begin{cor}\label{corharmint}
Let $X\in \mathcal{PR}^d$ where $d\geq 3$ and let $1\leq j<d$ and $k=j-1$. Then
\begin{align}\label{harmint}
\int_{\mathcal{E}_j^d} \tilde{\Xi}_{k,E}^{r,s}(X\cap E) \di E {}&= a_{s,j,d}c_{d-2,j-1}\sigma_{d-1}\Xi_{d-1}^{r,s} (X).
\end{align}
\end{cor}

\begin{rem}
The result in Corollary \ref{corharmint} resembles \cite[Corollary 6.1]{bernighug}. In  \cite{bernighug}, the tensors on the left hand side of \eqref{harmint} are computed with respect to the measures $\Lambda_k(X\cap E,\cdot)$, i.e.\ with $X\cap E$ considered as a subset of $\R^d$ rather than $E$. In this setting, a formula like \eqref{harmint} holds for all $k<j-1$ when $r=0$. We are able to show the formula for all $r\geq 0$, but only for $k=j-1$.  We leave it open whether a similar formula holds for $k<j-1$.

In applications, it is simpler and more natural to consider $X\cap E$ as a subset of $E$ when computing a tensor. On the other hand, the tensor $\tilde{\Xi}_{j-1,E}^{r,s}(X\cap E)$ is not a harmonic tensor on $E$ and therefore the formula \eqref{harmint} is less natural than the analogue in \cite{bernighug}.
\end{rem}

\begin{rem}
Corollary \ref{corharmint} can be used to define an unbiased estimator for the harmonic tensor $\Xi_{d-1}^{r,s} (X)$ if $a_{s,j,d}$ is non-zero. It can be shown that $a_{s,1,d}\neq 0$ if and only if $s$ is even or $s=1$. The question whether $a_{s,j,d}\neq 0$ for $j>1$ is more involved. It was answered in \cite{mk} in certain cases.  

Suppose we want to express the Minkowski tensor $\Phi_{d-1}^{r,s} (X)$ as a Crofton integral. In the case $r=0$ and $j=1$, this was done in \cite{astridmarkus}. In the general situation, we write $\Phi_{d-1}^{r,s} (X)$ as linear combination of the harmonic Minkowski tensors $Q^p \Xi_{d-1}^{r,s-2p} (X)$. This is possible because $u^s$ can be written as a linear combination of the tensors $Q^p H^{s-2p}_d(u)$ (see \cite[Proposition 4.10]{bernighug} when $r=0$). 
If the constants $a_{s-2p,j,d}$ are non-zero for all $p\geq 0$, then we can use \eqref{harmint} to express  $\Phi_{d-1}^{r,s}(X)$ as a Crofton integral. 
\end{rem}

\appendix

\section*{Appendix}
In this Appendix, we present proofs of  Proposition \ref{linkombprop} and Theorem \ref{rotational}. Since the proofs use the theory of normal cycles for sets of positive reach, we first recall the definition and basic properties. Details can be found in \cite{zahle}.

Let $X\in \mathcal{PR}^d$ be a compact set of positive reach. Then $\nor X $ is a compact $(d-1)$-rectifiable set. The oriented tangent space at $(x,n)\in \nor X$ can be identified with the simple $(d-1)$-vector 
\begin{equation*}
a_X(x,n) = \bigwedge_{1\leq i\leq d-1} \Big( \tfrac{1}{\sqrt{1+\kappa_i(x,n)^2}} a_i(x,n), \tfrac{\kappa_i(x,n)}{\sqrt{1+\kappa_i(x,n)^2}} a_i(x,n)\Big) \in \wedge_{d-1}(\R^d \times TS^{d-1}),
\end{equation*}
where $a_i$ and $\kappa_i$ are as in \eqref{awedge}. If $X\in \mathcal{PR}(L)$ is considered as a subset of $L\in \mathcal{L}_j^d$, then we denote  its normal bundle inside $\Sigma^L $ by $\nor^L X$ and the tangent $(j-1)$-vector field of $\nor^L X$ by $a_X^L\in \wedge_{j-1} (T \Sigma^L) $.

The normal cycle of $X\in \mathcal{PR}^d $ is a $(d-1)$-current $N_X$ that acts on a differential $(d-1)$-form $\omega$ in $\R^d \times S^{d-1}$ by
\begin{equation*}
N_X(\omega) = \int_{\nor X} \langle a_X, \omega \rangle\, \di \Ha^{d-1}.
\end{equation*}
The normal cycle is a cycle, i.e.\ it vanishes if $\omega$ is a coboundary.

The Lipschitz-Killing curvature form on $\R^d\times S^{d-1}$  is the $(d-1)$-form given in the coordinates $x_1,\dots,x_d,n_1,\dots,n_d$ on $\R^d \times \R^d$ by
\begin{align*}
\rho_k = \frac{1}{k!(d-k-1)!\sigma_{d-k}} \sum_{\sigma \in S_{d} } \sgn(\sigma) n_{\sigma(d) }\, \di x_{\sigma(1)}\wedge \dotsm \wedge \di x_{\sigma(k)} \wedge \di n_{\sigma(k+1)}\wedge \dotsm \wedge \di n_{\sigma(d-1)},
\end{align*}
where $S_d$ is the group of permutations of $\{1,\dots,d\}$ and $\sgn (\sigma)$ is the sign of $\sigma \in S_d$.
For $L\in \Li_j^d$, a relative Lipschitz-Killing $(j-1)$-curvature form $\rho_k^L$ on $L\times S^{j-1}(L)$ is defined in a similar way.
Suppose $X\in \mathcal{PR}(L)$. Integration of a  locally bounded measurable function $\psi: \Sigma^L \to \R$ with respect to $\Lambda_k^L(X;\cdot)$ is then given by
\begin{equation*}
\Psi^\psi_{k,L} (X)  =\int_{\Sigma^L} \psi(x,n)\, \Lambda^L_k (X;\di (x,n) )= \int_{\nor^L X } \psi(x,n) \langle a_{X}^L , \rho_k^L \rangle\, \Ha^{d-1}(\di (x,n)).  
\end{equation*}

\begin{proof}[Proof of Proposition \ref{linkombprop}]
The identity \eqref{trivialcase} follows easily from 
\begin{equation*}
Q= n^2 + \sum_{i=1}^{d-1} a_i^2.
\end{equation*}
Both \eqref{linkomb} and \eqref{linkomb2} follow inductively  from the identity
\begin{align}\label{induction}
\mathds{1}_{\{r\geq 0\}}{}&\Phi_{k}^{r,s-2,1}(X) \\ \nonumber
&= \mathds{1}_{\{k>1,s>2\}}\Phi_{k-1}^{r+1,s-3,1}(X) + 2\pi (s-1)\Phi_{k-1}^{r+1,s-1}(X) - \mathds{1}_{\{s>2\}} Q \Phi_{k-1}^{r+1,s-3}(X) 
\end{align}
for $s\geq 2$, $r\geq -1$, and $1 \leq k\leq d-1$.
In the case of \eqref{linkomb2}, the induction start is \eqref{trivialcase}.
Equation \eqref{r0} is the special case $r=-1$.

The Minkowski tensor $\Phi_{k}^{r,s}(X)$ is given by applying the normal cycle $N_X$ to a $(d-1)$-form $\rho_{k}^{r,s} : \wedge_{d-1} (\R^d \times \R^d) \to \mathbb{T}^{r+s}$. Choose  an oriented orthonormal basis $e_1,\dots,e_d$ for $\R^d$. Then the coordinate of $\rho_{k}^{r,s}$ corresponding to the basis element $e_{i_1}\odot \dotsm \odot e_{i_{r+s}} \in \mathbb{T}^{r+s}$ is 
\begin{align*}
\rho_{k,i_1,\dots,i_{r+s}}^{r,s} ={}& \frac{1}{k!(d-k-1)!\omega_{d-k+s} r!s!}x_{i_1} \dots x_{i_r}n_{i_{r+1}}\dots n_{i_{r+s}}\\
&\times \sum_{\sigma \in S_{d} } \sgn(\sigma) n_{\sigma(d) }\, \di x_{\sigma(1)}\wedge \dotsm \wedge \di x_{\sigma(k)} \wedge \di n_{\sigma(k+1)}\wedge \dotsm \wedge \di n_{\sigma(d-1)}. 
\end{align*}
Here and throughout the proof, symmetrization in the indices $i_1,\dots,i_{r+s}$ is understood when we write the coordinates of a $\mathbb{T}^{r+s}$-valued tensor.
Similarly, it can be shown \cite[Section 4]{hugschneider} that  $\Phi_{k}^{r,s-2,1}(X)$ is given by applying the normal cycle to the $(d-1)$-form with coordinates 
\begin{align*}
&\eta_{k,i_1,\dots,i_{r+s}}^{r,s} = \frac{1}{(k-1)!(d-k-1)!\omega_{d-k+s-2} r!(s-2)!}x_{i_1}\dotsm x_{i_r} n_{i_{r+1}}\dots n_{i_{r+s-2}}\\
&\times \, \di x_{i_{r+s-1}}\wedge\sum_{\sigma \in S_{d}, \sigma(1)=i_{r+s} }\sgn(\sigma) n_{\sigma(d) }\, \di x_{\sigma(2)}\wedge \dotsm \wedge \di x_{\sigma(k)} \wedge \di n_{\sigma(k+1)}\wedge \dotsm \wedge \di n_{\sigma(d-1)}. 
\end{align*}

We first show \eqref{induction} in the case $r\geq 0$. Define the $\mathbb{T}^{r+s}$-valued $(d-2)$-form
\begin{align*}
\omega_{k,i_1,\dots,i_{r+s}}^{r,s} = {}& \frac{1}{(k-1)!(d-k-1)!\omega_{d-k+s-2} (r+1)!(s-2)!}x_{i_1}\dotsm x_{i_r}x_{i_{r+s-1}} n_{i_{r+1}}\dots n_{i_{r+s-2}}\\
&\times \sum_{\sigma \in S_{d}, \sigma(1)=i_{r+s} } \sgn(\sigma) n_{\sigma(d) }\, \di x_{\sigma(2)}\wedge \dotsm \wedge \di x_{\sigma(k)} \wedge \di n_{\sigma(k+1)}\wedge \dotsm \wedge \di n_{\sigma(d-1)}. 
\end{align*}
Since $N_X$ vanishes on coboundaries, it suffices to show that pointwise
\begin{equation}\label{formid}
\big\langle a_X,\di \omega_k^{r,s} \big\rangle = \big\langle a_X,\eta_{k}^{r,s} - \mathds{1}_{\{k>1,s>2\}}\eta_{k-1}^{r+1,s-1} - 2\pi (s-1)\rho_{k-1}^{r+1,s-1} + \mathds{1}_{\{s>2\}}Q \rho_{k-1}^{r+1,s-3}\big\rangle. 
\end{equation}
It is straight-forward to check that all differential forms in \eqref{formid} are $\SO(d)$-covariant. (We say that $\omega$ is $\SO(d)$-covariant if for all $\theta \in \SO(d)$,
\begin{equation*}
\big\langle \wedge_{i=1}^{d-1} (\theta_*v_i,\theta_*w_i) ,  \omega_{(\theta(x),\theta(n))} (\theta u_1,\dots , \theta u_{r+s}) \big\rangle = \big\langle \wedge_{i=1}^{d-1} (v_i,w_i) ,  \omega_{(x,n)} (u_1,\dots ,  u_{r+s}) \big\rangle,
\end{equation*}
for all $(v_i,w_i)\in \R^d \times \R^d$ and $u_1,\dots u_{r+s}\in \R^d $.)
It is therefore enough to show \eqref{formid} when $(x,n)=(x,e_d)$. At this point, forms of the type $\di n_d \wedge \xi$ and $\di x_d \wedge \xi$ vanish on $a_X$.
 We first compute:
\begin{align}\nonumber
\di {}&\omega_{k,i_1,\dots,i_{r+s}}^{r,s}= \eta_{k,i_1,\dots,i_{r+s}}^{r,s}\\ \nonumber
&+ \mathds{1}_{\{s>2\}} \sum_{j=r+1}^{r+s-2} \frac{x_{i_1} \dotsm x_{i_r}x_{i_{r+s-1}} n_{i_{r+1}} \dotsm \hat{n}_{i_j} \dotsm n_{i_{r+s-2}}}{(k-1)!(d-k-1)!\omega_{d-k+s-2} (r+1)!(s-2)!}\sum_{\sigma \in S_{d} } \sgn(\sigma)\\ \label{a}
&\quad \times  \delta_{ \sigma(1),i_{r+s}} n_{\sigma(d) }\, \di {n_{i_j}} \wedge  \di x_{\sigma(2)}\wedge \dotsm \wedge \di x_{\sigma(k)} \wedge \di n_{\sigma(k+1)}\wedge \dotsm \wedge \di n_{\sigma(d-1)}\\ \nonumber
&+ \frac{x_{i_1} \dotsm x_{i_r}x_{i_{r+s-1}} n_{i_{r+1}} \dotsm n_{i_{r+s-2}}}{(k-1)!(d-k-1)!\omega_{d-k+s-2} (r+1)!(s-2)!}\sum_{\sigma \in S_{d} } \sgn(\sigma)  \\ \label{b}
&\quad \times \delta_{ \sigma(1),i_{r+s}}\, \di n_{\sigma(d) }\wedge \di x_{\sigma(2)}\wedge \dotsm \wedge \di x_{\sigma(k)} \wedge \di n_{\sigma(k+1)}\wedge \dotsm \wedge \di n_{\sigma(d-1)}.
\end{align}
Here the Kronecker $\delta$-notation has been used and a "$\hat{\ }$" indicates that a factor is left out. We see that \eqref{a} equals
\begin{align} \nonumber
& \mathds{1}_{\{s>2\}} \sum_{j=r+1}^{r+s-2} \frac{x_{i_1} \dotsm x_{i_r}x_{i_{r+s-1}} n_{i_{r+1}} \dotsm \hat{n}_{i_j} \dotsm n_{i_{r+s-2}} }{(k-1)!(d-k-1)!\omega_{d-k+s-2} (r+1)!(s-2)!}\delta_{i_j,i_{r+s}}\sum_{\sigma \in S_{d}}  \sgn(\sigma)\delta_{\sigma(1),i_{r+s} } \\ \label{c}
&\quad \times n_{\sigma(d) } \, \di n_{\sigma(1)} \wedge \di x_{\sigma(2)}\wedge \dotsm \wedge \di x_{\sigma(k)} \wedge \di n_{\sigma(k+1)}\wedge \dotsm \wedge \di n_{\sigma(d-1)}\\ \nonumber
&- \mathds{1}_{\{s>2\}} \sum_{j=r+1}^{r+s-2} \frac{x_{i_1} \dotsm x_{i_r}x_{i_{r+s-1}} n_{i_{r+1}} \dotsm \hat{n}_{i_j} \dotsm n_{i_{r+s-2}}}{(k-1)!(d-k-1)!\omega_{d-k+s-2} (r+1)!(s-2)!} \sum_{\sigma \in S_{d}} \sgn(\sigma)\delta_{ \sigma(1),i_{r+s} } \sum_{m=2}^k  \\ \label{d}
&\quad\times \delta_{i_j,\sigma(m)} n_{\sigma(d) }\, \di {x_{i_j}}\wedge  \di x_{\sigma(2)}\wedge  \dotsm \di n_{\sigma(m)} \dotsm \wedge \di x_{\sigma(k)} \wedge \di n_{\sigma(k+1)}\wedge \dotsm \wedge \di n_{\sigma(d-1)}\\ \nonumber
&+ \mathds{1}_{\{s>2\}} \sum_{j=r+1}^{r+s-2} \frac{x_{i_1} \dotsm x_{i_r}x_{i_{r+s-1}} n_{i_{r+1}} \dotsm \hat{n}_{i_j} \dotsm n_{i_{r+s-2}}}{(k-1)!(d-k-1)!\omega_{d-k+s-2} (r+1)!(s-2)!} \sum_{\sigma \in S_{d} } \sgn(\sigma) \delta_{\sigma(1),i_{r+s}}\\ \label{e}
&\quad \times \delta_{i_j,\sigma(d)} n_{\sigma(d) } \, \di {n_{\sigma(d)}} \wedge \di x_{\sigma(2)}\wedge \dotsm \wedge \di x_{\sigma(k)} \wedge \di n_{\sigma(k+1)}\wedge \dotsm \wedge \di n_{\sigma(d-1)}.
\end{align}
Clearly \eqref{e} vanishes at $(x,e_d)$ when evaluated on $a_X$. Moreover, \eqref{c} equals
\begin{align}\nonumber
& Q \rho_{k-1}^{r+1,s-3} \mathds{1}_{\{s>2\}} \\ \nonumber
&- \mathds{1}_{\{k>1,s>2\}} \sum_{j=r+1}^{r+s-2} \frac{x_{i_1} \dotsm x_{i_r}x_{i_{r+s-1}} n_{i_{r+1}} \dotsm \hat{n}_{i_j} \dotsm n_{i_{r+s-2}}}{(k-2)!(d-k)!\omega_{d-k+s-2} (r+1)!(s-2)!} \delta_{i_j,i_{r+s}} \sum_{\sigma \in S_{d}}  \sgn(\sigma) \\ \label{f}
&\quad \times \delta_{\sigma(1),i_{r+s} } n_{\sigma(d) } \, \di x_{\sigma(1)} \wedge  \di x_{\sigma(2)}\wedge \dotsm \wedge \di x_{\sigma(k-1)} \wedge \di n_{\sigma(k)}\wedge \dotsm \wedge \di n_{\sigma(d-1)}\\ \nonumber
&- \mathds{1}_{\{s>2\}} \sum_{j=r+1}^{r+s-2} \frac{x_{i_1} \dotsm x_{i_r}x_{i_{r+s-1}} n_{i_{r+1}} \dotsm \hat{n}_{i_j} \dotsm n_{i_{r+s-2}} }{(k-1)!(d-k-1)!\omega_{d-k+s-2} (r+1)!(s-2)!}\delta_{i_j,i_{r+s}}\sum_{\sigma \in S_{d} } \sgn(\sigma)  \\ \label{g}
&\quad \times \delta_{\sigma(d),i_{r+s}} n_{\sigma(d) }\,  \di x_{\sigma(1)}\wedge \di x_{\sigma(2)}\wedge \dotsm \wedge \di x_{\sigma(k-1)} \wedge \di n_{\sigma(k)}\wedge \dotsm \wedge \di n_{\sigma(d-1)},
\end{align}
and \eqref{d} + \eqref{f} equals
\begin{align}\nonumber
 &- \mathds{1}_{\{s>2,k>1\}} \sum_{j=r+1}^{r+s-2} \frac{(k-1)x_{i_1} \dotsm x_{i_r}x_{i_{r+s-1}} n_{i_{r+1}} \dotsm \hat{n}_{i_j} \dotsm n_{i_{r+s-2}}}{(k-1)!(d-k-1)!\omega_{d-k+s-2} (r+1)!(s-2)!}\sum_{\sigma \in S_{d} }\sgn(\sigma)  \\ \nonumber
&\quad\times \delta_{\sigma(1),i_{r+s}} \delta_{i_j,\sigma(k)} n_{\sigma(d) }\, \di {x_{i_j}}\wedge \di x_{\sigma(2)}\wedge \dotsm \di x_{\sigma(k-1)} \wedge \di n_{\sigma(k)} \wedge \dotsm \wedge \di n_{\sigma(d-1)}\\ \nonumber
&- \mathds{1}_{\{s>2,k>1\}}\sum_{j=r+1}^{r+s-2} \frac{(k-1)x_{i_1} \dotsm x_{i_r}x_{i_{r+s-1}} n_{i_{r+1}} \dotsm \hat{n}_{i_j} \dotsm n_{i_{r+s-2}}}{(k-1)!(d-k)!\omega_{d-k+s-2} (r+1)!(s-2)!} \sum_{\sigma \in S_{d} } \sgn(\sigma) \\ \nonumber
&\quad\times \delta_{\sigma(1),i_{r+s}} \delta_{i_j,\sigma(1)} n_{\sigma(d) } \, \di x_{i_j}\wedge  \di x_{\sigma(2)}\wedge \dotsm \wedge \di x_{\sigma(k-1)} \wedge \di n_{\sigma(k)}\wedge \dotsm \wedge \di n_{\sigma(d-1)}\\ \nonumber
&=-\mathds{1}_{\{s>2,k>1\}} \eta_{k-1,i_1,\dots,i_{r+s}}^{r+1,s-1}\\ \nonumber
&+ \mathds{1}_{\{s>2,k>1\}}\sum_{j=r+1}^{r+s-2} \frac{x_{i_1} \dotsm x_{i_r}x_{i_{r+s-1}} n_{i_{r+1}} \dotsm \hat{n}_{i_j} \dotsm n_{i_{r+s-2}} }{(k-2)!(d-k-1)!\omega_{d-k+s-2} (r+1)!(s-2)!}\sum_{\sigma \in S_{d}}\sgn(\sigma) \\
&\quad\times \delta_{\sigma(1),i_{r+s} } \delta_{\sigma(d),i_j}  n_{\sigma(d) } \, \di x_{i_j}\wedge  \di x_{\sigma(2)}\wedge \dotsm \wedge \di x_{\sigma(k-1)} \wedge \di n_{\sigma(k)}\wedge \dotsm \wedge \di n_{\sigma(d-1)}. \label{h}
\end{align}
Also \eqref{h} vanishes at $(x,e_d)$. It remains to evaluate \eqref{b}+\eqref{g} at $(x,e_d)$. This yields %
\begin{align*}
& \frac{x_{i_1} \dotsm x_{i_r}x_{i_{r+s-1}} \delta_{{i_{r+1}}, \dots ,i_{r+s-2},{i_{r+s}},d}}{(k-1)!(d-k-1)!\omega_{d-k+s-2} (r+1)!(s-2)!}\\
&\quad\times \sum_{\sigma \in S_{d} } (-\sgn(\sigma)) \delta_{\sigma(d),i_{r+s}}\,  \di x_{\sigma(1)}\wedge \dotsm \wedge \di x_{\sigma(k-1)} \wedge \di n_{\sigma(k)}\wedge \dotsm \wedge \di n_{\sigma(d-1)}\\
&- \mathds{1}_{\{s>2\}} (s-2) \frac{x_{i_1} \dotsm x_{i_r}x_{i_{r+s-1}} \delta_{{i_{r+1}}, \dots ,i_{r+s-2},{i_{r+s}} ,d}}{(k-1)!(d-k)!\omega_{d-k+s-2} (r+1)!(s-2)!}\\
&\quad\times \sum_{\sigma \in S_{d} } \sgn(\sigma) \delta_{\sigma(d),i_{r+s} } \, \di x_{\sigma(1)} \wedge \dotsm \wedge \di x_{\sigma(k-1)} \wedge \di n_{\sigma(k)}\wedge \dotsm \wedge \di n_{\sigma(d-1)}\\
&= -2\pi (s-1)\rho_{k-1,i_1,\dots,i_{r+s}}^{r+1,s-1}.
\end{align*}
Putting things together yields \eqref{formid}.

To check the remaining case $r=-1$, $1\leq k<d-1$, in \eqref{induction}, it is enough to show
\begin{equation*}
\big\langle a_X, \di \tilde{\omega}_{k}^{0,s}  \big\rangle = \big\langle a_X, -\mathds{1}_{\{k>1,s>2\}}\eta_{k}^{0,s-2} - 2\pi s\rho_{k}^{0,s} + \mathds{1}_{\{s>2\}}Q \rho_{k}^{0,s-2}\big\rangle =0,
\end{equation*}
where
\begin{align*}
\tilde{\omega}_{k,i_1,\dots,i_{s}}^{0,s} = {}& \frac{1}{k!(d-k-2)!\omega_{d-k+s-2} (s-1)!}n_{i_{1}}\dots n_{i_{s-1}}\\
&\times \sum_{\sigma \in S_{d}, \sigma(1)=i_{s} } \sgn(\sigma) n_{\sigma(d) }\, \di x_{\sigma(2)}\wedge \dotsm \wedge \di x_{\sigma(k+1)} \wedge \di n_{\sigma(k+2)}\wedge \dotsm \wedge \di n_{\sigma(d-1)}. 
\end{align*}
The computations are similar, but slightly simpler.
\end{proof}

\begin{proof}[Proof of Theorem \ref{rotational}]
We first show that $ \Psi^\psi_{k,L}(X\cap L)$ is integrable with respect to $L$. We have
\begin{align*}
&\int_{\Li_j^d}|\Psi^\psi_{k,L}(X\cap L) |\,\di L\\
{}&=\frac{1}{\sigma_j}\int_{\Li_{j-1}^d} \int_{S^{d-j}(L^\perp)} \bigg|\int_{\nor^{L^z} (X \cap L^z)} \psi(x,n) \langle a_{X\cap L^z}^{L^z}, \rho_k^{L^{z}} \rangle \,\Ha^{j-1}(\di (x,n))\bigg|\, \di z \, \di L \\
&\leq  \frac{\sup_{(x,n)\in \nor X}\psi(x,n)}{\sigma_j}\int_{\Li_{j-1}^d} \int_{S^{d-j}(L^\perp)}\int_{\nor^{L^z} (X \cap L^z)}| \langle a_{X\cap L^z}^{L^z}, \rho_k^{L^z} \rangle |\,\Ha^{j-1}(\di (x,n))\,\di z \,\di L .
\end{align*}
 Fix $L\in \Li_{j-1}^d$ such that for $\Ha^{j-1}$-almost all $(x,n)\in \nor X$ we have $x\notin L$ and for almost all $z\in S^{d-j}(L^\perp) $: 
\begin{equation}\label{condition}
\text{There is no } (x,n)\in \nor X \text{ with } x\in L^z\text{ and } n\perp L^z.
\end{equation}
 Almost all $L\in \Li_{j-1}^d$ satisfy these two conditions by \cite[Lemma 5]{evarataj} and the definition of $\widetilde{\mathcal{PR}}^d$.

Define two functions $f$ and $g$ as in \cite[Section 7]{evarataj} by
\begin{align*}
f:{}& \nor X \backslash \{(x,n) \in \nor X  \mid x\in L \text{ or }n\perp L^x\}\to \R^d \times S^{d-j}(L^z),\\
& f(x,n)=(x,\pi(n|L^x)), \\
g:{}& \nor X \backslash \{(x,n) \in \nor X \mid x\in L\}\to S^{d-j}(L^\perp),\\
& g(x,n)=\pi(x|L^\perp),
\end{align*}
and for $z\in S^{d-j}(L^\perp)$, define
\begin{align*}
h_z:{}&  L \times S^{d-1} \backslash \{(x,n)\in L \times S^{d-1} \mid n \perp L^z \} \to L \times S^{j-1}(L),\\ 
& h_z(x,n)=(x,\pi(n|L^z)) .
\end{align*}
It follows from \cite[Theorem 4.10 (3)]{federer2} that for  $z\in S^{d-j}(L^\perp)$ satisfying \eqref{condition}, we have
\begin{align*}
\nor^{L^z} (X \cap L^z) {}&= \{(x,\pi(n|L^z)) \mid (x,n )\in \nor X , x\in L^z\} \\
&= f\big(g^{-1}(z)\cup g^{-1}(-z)\big) \cup h_z(N),
\end{align*}
where $N = \{(x,n)\in \nor X \mid x\in L\}$ has $\Ha^{j-1}$-measure 0 by assumption. Because of condition \eqref{condition}, $h_z$ is locally Lipschitz on $N$, and hence $\Ha^{j-1}(h_z(N)) = 0$.

It was shown in \cite[Lemma 3]{evarataj} that for almost all  $z$,
\begin{equation}\label{measurezero}
\Ha^{j-1}\big((x,n)\in f(g^{-1}(z))\mid \# \big( f^{-1}(x,n)\cap g^{-1}(z)  \big) >1\big) = 0,
\end{equation}
where $\#$ denotes cardinality.
It follows from \cite[Proof of Lemma 4]{evarataj} that for almost all $z$ it holds for $\Ha^{j-1}$-almost all $(x,n) \in \nor^{L^z} (X \cap L^z)$ that there is a unique $y\in g^{-1}(z)\cup g^{-1}(-z)$  with $(x,n) = f(y)$ and
\begin{equation*}
a_{X\cap L^z}^{L^z} (x,n) = f_\sharp (\zeta(y)) / J_{j-1}(f_{\mid g^{-1}(z)\cup g^{-1}(-z)})(y),
\end{equation*}
where 
\begin{equation*}
\zeta = (a_X \llcorner  g^\sharp \Omega_{d-j} )/ J_{d-j} g
\end{equation*}
and $\Omega_{d-j}$ is an orientation form on $S^{d-j}(L^\perp)$ as in \cite[p.\ 552]{evarataj}.

Since $f$ is locally Lipschitz on $g^{-1}(z)$ for almost all $z\in S^{d-j}(L^\perp)$, 
the area and coarea formulae together with \eqref{measurezero} yield
\begin{align}\nonumber
\int_{\Li_{j-1}^d}{}& \int_{S^{d-j}(L^\perp)} \int_{\nor^{L^z} (X \cap L^z)}| \langle a_{X\cap L^z}^{L^z}, \rho_k^{L^z} \rangle |\Ha^{j-1}(\di (x,n)) \,\di z\, \di L\\ \nonumber
 {}&= 2\int_{\Li_{j-1}^d} \int_{S^{d-j}(L^\perp)}\int_{f(g^{-1}(z))}| \langle a_{X\cap L^z}^{L^z}, \rho_k^{L^z} \rangle |\Ha^{j-1}(\di (x,n))\, \di z\, \di L\\ \nonumber
&=2\int_{\Li_{j-1}^d} \int_{S^{d-j}(L^\perp)} \int_{g^{-1}(z)} | \langle \zeta , f^{\sharp} \rho_k^{L^x} \rangle |  \Ha^{j-1}(d(x,n))\, \di z\, \di L  \\ \label{finite}
&=2\int_{\Li_{j-1}^d} \int_{\nor X}  | \langle a_X ,g^\sharp \Omega_{d-j}\wedge  f^{\sharp} \rho_k^{L^x} \rangle |\,  \Ha^{d-1}(\di (x,n)) \, \di L.  
\end{align}
We must show that \eqref{finite} is finite.
In \cite[Section 7]{evarataj} it is shown that
\begin{align}\label{formvalue}
 \langle a_X, g^\sharp \Omega_{d-j} \wedge f^\sharp \rho_k^{L^x}  \rangle = {}& \frac{1}{\sigma_{j-k}} \frac{1}{|p(x|L^\perp)|^{d-j} |p(n | L^x)|^{j-k}}   \\
& \times \sum_{|I|=j-1-k} \frac{\prod_{i\in I} \kappa_i(x,n)}{\prod_{i=1}^{d-1} \sqrt{1+\kappa_i(x,n)^2}} \G(L^x, A_I(x,n))^2. \nonumber
\end{align}
Since $\G(L^x, A_I(x,n)) \leq |p(n|L^x)|$, there is a $C>0$ such that
\begin{align*}
\bigg| \frac{1 }{|p(x|L^\perp)|^{d-j}|p(n | L^x)|^{j-k}} {}&  \sum_{|I|=j-1-k} \frac{\prod_{i\in I} \kappa_i(x,n)}{\prod_{i=1}^{d-1} \sqrt{1+\kappa_i(x,n)^2}} \G(L^x, A_I(x,n))^2 \bigg| \\&\leq  \frac{C}{|p(x|L^\perp)|^{d-j}|p(n | L^x)|^{j-k-2}}. 
\end{align*}
It follows from e.g. \cite[(1)]{zahle} that $\Ha^{d-1}(\nor X)<\infty$ when $X$ is compact, so it is enough to show that 
\begin{align}\label{bound}
\int_{\Li_{j-1}^{d}}{}& |p(x|L^\perp)|^{j-d}|p(n | L^x)|^{2-j+k} \di L
\end{align}
is uniformly bounded on $\nor X$.
By assumption, there is an $\eps >0$ such that $|x|\geq \eps $ on $\nor X$. 
The $(j-1)(d-j)$-Jacobian of the map $\Li_{j-1}^d \backslash \Li_{j-1}^x \to \Li^{x^\perp}_{j-1} $ given by $L \mapsto p(L|x^\perp)$ was computed in \cite[Lemma 6]{evarataj} to be $|x|^{d-j}|p(x|L^\perp )|^{j-d}$. Letting $w=\pi(n|x^\perp)$, the coarea formula yields
\begin{align*}
\int_{\Li_{j-1}^{d}} |p(x|L^\perp)|^{j-d}|p(n | L^x)|^{2-j+k}  \di L{}&\leq \eps^{j-d} \int_{\Li_{j-1}^{x^\perp}} |p(n | L^x)|^{2-j+k}\, \di L\\
&= \eps^{j-d}\int_{\Li_{j-1}^{x^\perp}} (1-\alpha^2 +\alpha^2 |p( w | L)|^2)^{\frac{2-j+k}{2}} \,\di L,
\end{align*}
where $\alpha=\alpha(x,n)=\sin \angle (x,n)$. If $j-k \leq 2$ or $\alpha=0$, this is clearly bounded. Assume $j-k>2 $ and $\alpha \neq 0$. Using \eqref{coarea} and \eqref{sphereform}, we then have 
\begin{align*}
\int_{\Li_{j-1}^{x^\perp}}{}& |p(n | L^x)|^{2-j+k} \,\di L \\
&= c_{d-3,j-2}\int_{S^{d-2}(x^\perp)}\mathds{1}_{\{\langle u , w \rangle> 0\}}\frac{\langle u , w \rangle^{j-2} }{(1-\langle u , w \rangle^2)^{\frac{j-2}{2}}} (1-\alpha^2 +\alpha^2 \langle u , w \rangle^2)^{\frac{2-j+k}{2}}\, \di u\\
{}&= c_{d-3,j-2}\sigma_{d-2}\int_0^1 {t}^{j-2}(1-t^2)^{\frac{d-j-2}{2}}(1-\alpha^2 +\alpha^2 t^2)^{\frac{2-j+k}{2}}\, \di t\\
&\leq c_{d-3,j-2}\sigma_{d-2}\int_0^1 (1-t^2)^{\frac{d-j-2}{2}} t^{{k}} \,\di t.
\end{align*}
Since $2 < j\leq d-1$, this is finite and hence \eqref{finite} is finite.

The finiteness of \eqref{finite} allows us to apply the area and coarea to the rotational integral. A computation similar to \eqref{finite} yields 
\begin{align*}
\int_{\Li_{j}^d}{}& \Psi_{k,L}^{\psi}(X\cap L)\,  \di L \\
{}&= \frac{2}{\sigma_j}\int_{\Li_{j-1}^d} \int_{S^{d-j}(L^\perp)} \int_{f(g^{-1}(z))} \psi(L^z,x,n) \langle a_{X \cap L^z}^{L^z} , \rho_k^{L^z}\rangle \, \Ha^{j-1}(\di (x,n)) \, \di z\, \di L\\
&=\frac{2}{\sigma_j}\int_{\Li_{j-1}^d} \int_{\nor X} \psi(L^x,f(x,n)) \langle a_X , g^\sharp \Omega_{d-j} \wedge f^\sharp \rho_k^{L^x}\rangle  \, \Ha^{d-1}(\di (x,n))\, \di L.  
\end{align*}
Using again that \eqref{finite} is finite, we can apply Fubini's theorem. Inserting \eqref{formvalue}, we get 
\begin{align*}
\int_{\Li_j^d} \Psi^\psi_{k,L}(X{}&  \cap L)\, \di L= \frac{2}{\sigma_j\sigma_{j-k}}\int_{\nor X} \int_{\Li_{j-1}^d}  \frac{\psi(L^x,x,\pi(n | L^x))}{|p(x|L^\perp)|^{d-j} |p(n | L^x)|^{j-k}}   \\
& \times \sum_{|I|=j-1-k} \frac{\prod_{i\in I} \kappa_i(x,n)}{\prod_{i=1}^{d-1} \sqrt{1+\kappa_i(x,n)^2}} \G(L^x, A_I(x,n))^2  \, \di L\, \Ha^{d-1}(\di (x,n))  . 
\end{align*}
Since \eqref{bound} was finite, we can apply the coarea formula for the map $L\mapsto p(L|x^\perp)$ once again to obtain the claim of Theorem \ref{rotational}:
\begin{align*}
\int_{\Li_j^d}\Psi^\psi_{k,L}(X{}&\cap L) \, \di L= \frac{1}{\sigma_{j-k}} \int_{\nor X} \int_{\Li_{j-1}^{x^\perp}}  \frac{\psi(L^x,x,\pi(n | L^x))}{|x|^{d-j}|p(n | L^x)|^{j-k}} \\
& \times \sum_{|I|=j-1-k} \frac{\prod_{i\in I} \kappa_i(x,n)}{\prod_{i=1}^{d-1} \sqrt{1+\kappa_i(x,n)^2}} \G(L^x, A_I(x,n))^2  \, \di L \, \Ha^{d-1}(\di (x,n)). 
\end{align*}
\end{proof}

\end{document}